\documentclass[a4paper,11pt,british,oneside]{article}

\usepackage[utf8]{inputenc}
\usepackage[square,comma]{natbib}
\usepackage[dvipsnames,svgnames,x11names,hyperref]{xcolor}
\usepackage{amsmath,amssymb,csquotes,enumerate,dsfont,babel,mathtools,textgreek,hyphenat,needspace}
\usepackage[amsmath,amsthm,thmmarks,hyperref]{ntheorem}
\definecolor{linkcol}{rgb}{0,0.2,0.6}
\definecolor{citecol}{rgb}{0,0.4,0}
\definecolor{green}{rgb}{0.67,0.77,0}
\usepackage[pdfpagelabels,bookmarks,colorlinks,linkcolor=linkcol,citecolor=citecol,urlcolor=black]{hyperref}
\newcounter{thcounter}
\numberwithin{thcounter}{section}
\numberwithin{equation}{section}

\theoremstyle{break}
\newtheorem{lemma}[thcounter]{Lemma}
\newtheorem{proposition}[thcounter]{Proposition}
\newtheorem{theorem}[thcounter]{Theorem}
\newtheorem{corollary}[thcounter]{Corollary}
\theorembodyfont{\upshape}
\newtheorem{definition}[thcounter]{Definition}
\newtheorem{remark}[thcounter]{Remark}
\newtheorem{example}[thcounter]{Example}

\usepackage{wasysym,stmaryrd,tikz}
\usetikzlibrary{cd}

\setcitestyle{numbers}

\theoremnumbering{Alph}
\theorembodyfont{\itshape}
\newtheorem{theoremalpha}{Theorem}
\usepackage{geometry}
\begin{document}

\renewcommand{\bar}{\overline}
\renewcommand{\tilde}{\widetilde}
\renewcommand{\hat}{\widehat}
\makeatletter
\let\@temp\phi
\let\phi\varphi
\let\varphi\@temp

\let\@temp\epsilon
\let\epsilon\varepsilon
\let\varepsilon\@temp
\makeatother

\providecommand\numberthis{\addtocounter{equation}{1}\tag{\theequation}}
\providecommand{\CC}{\mathbb{C}}

\newcommand{\Tr}{\operatorname{Tr}}
\newcommand{\Hom}{\operatorname{Hom}}
\newcommand{\Irr}{\operatorname{Irr}}
\newcommand{\onb}{\operatorname{onb}}
\newcommand{\mult}{\operatorname{mult}}
\newcommand{\I}{\mathds{1}}
\newcommand{\id}{\mathrm{id}}
\newcommand{\End}{\operatorname{End}}
\newcommand{\Hilb}{\operatorname{Hilb}}
\newcommand{\HilbBimod}{\operatorname{HilbBimod}}
\newcommand{\Rep}{\operatorname{Rep}}
\newcommand{\tensorhat}{\mathbin{\hat{\otimes}}}
\newcommand{\emptyword}{\diameter}
\newcommand{\catid}{\mathbf{1}}
\newcommand{\qdim}{\mathrm{d}}

\begin{center}
    {\LARGE\boldmath\bf Some remarks on free products of rigid $C^*$-2-categories}

    \bigskip

    {\sc Matthias Valvekens\footnotemark}
    
    \bigskip
    
    {\small\sc KU~Leuven, Department of Mathematics\\ Celestijnenlaan 200B, B-3001 Leuven, Belgium}
\end{center}
\footnotetext{Email address: \texttt{maths@mvalvekens.be}.}

\begin{abstract}
    In this informal expository note, we present a universal, formulaic construction of the free product of rigid $C^*$-2-categories.
    This construction allows for a straightforward, purely categorical formulation of the free composition of subfactors and planar algebras considered by Bisch and Jones \cite{bisch-jones-interm-subfactors,jones-planar-algebras}.
    As an application, we explain the results of \cite{tarrago-wahl} on free wreath products of compact quantum groups in this categorical language.
\end{abstract}

\section{Introduction}

First studied in \cite{bisch-jones-fuss-catalan,bisch-jones-interm-subfactors}, the free composition of subfactors $N\subset M$ and $M\subset P$, along with the (closely related) free product of rigid $C^*$-tensor categories appear frequently in the literature on subfactors and quantum symmetries.

Given two rigid $C^*$-tensor categories $\mathcal{C}_1$ and $\mathcal{C}_2$, the free product is the rigid $C^*$-tensor category containing $\mathcal{C}_1$ and $\mathcal{C}_2$ with ``minimal'' relations.
Concretely, this is often stated as follows\footnote{Compare with the definition of the free composition of subfactors in \cite{bisch-jones-interm-subfactors}.}.
If $\mathcal{C}_1$ and $\mathcal{C}_2$ are given as subcategories of a third rigid $C^*$-tensor category $\mathcal{C}$, we say that $\mathcal{C}_1$ and $\mathcal{C}_2$ are \textit{free} inside $\mathcal{C}$ if any alternating tensor product of nontrivial irreducible objects in $\mathcal{C}_1$ and $\mathcal{C}_2$ remains irreducible.
If such irreducibles generate all of $\mathcal{C}$, we say that $\mathcal{C}$ is a free product of $\mathcal{C}_1$ and $\mathcal{C}_2$.
This ``working definition'' is sufficient for many purposes, but immediately raises two questions:
\begin{enumerate}[(i)]
    \item Does such a category $\mathcal{C}$ always exist?
    \item If so, in what sense is it unique?
\end{enumerate}
This also applies to the original definition given in \cite{bisch-jones-interm-subfactors} of the free composition of subfactors---see \cite[p.~94]{bisch-jones-interm-subfactors} in particular.
To some degree, these questions have been addressed and answered in the literature.
\begin{itemize}
    \item If for $i\in\{1,2\}$, $\mathcal{C}_i$ is the category of finite-dimensional unitary representations of a compact quantum group $\mathbb{G}_i$, then the free product of $\mathcal{C}_1$ and $\mathcal{C}_2$ is concretely given as the representation category of the free product compact quantum group $\mathbb{G}_1*\mathbb{G}_2$ in the sense of \cite{wang-free-products}.
    \item More generally, \cite{ghosh-jones-reddy} gives a diagrammatic construction of the free product of any two rigid $C^*$-tensor categories.
    \item The free composition of subfactors has been recast in planar algebra language in \cite{jones-planar-algebras}; see also \cite{tarrago-wahl}.
    \item In the formulation of standard invariants as paragroups, the free composition was defined by Gnerre in \cite{gnerre-thesis}.
\end{itemize}

In this note, we unify the above in a common framework by providing a construction for arbitrary free products of rigid $C^*$-2-categories.
While the core ideas are by no means new, we believe that our exposition is complementary to the material on free products published in the literature so far, for the following reasons.
\begin{itemize}
    \item Our free product construction comes with a universal property that guarantees both existence and uniqueness in the usual way.
    \item Our approach is entirely formulaic, and does not require a diagram calculus a priori.
    In particular, it is painless to generalise to free products with an arbitrary number of factors.
    Of course, the universal property guarantees that the diagram calculus of \cite{ghosh-jones-reddy} remains valid, but it is not part of the construction.
    \item The framework of rigid $C^*$-2-categories is flexible enough to also encode the free product of subfactor planar algebras.
    While closely related to the free product of rigid $C^*$-tensor categories, it is not quite the same, and the precise difference is immediately apparent in 2-category language.
    We use this to provide some categorical clarification with regard to the results of \cite{tarrago-wahl} on free wreath products of compact quantum groups.
\end{itemize}

Concretely, the goal of this note is to give a proof of the following theorem.
\begin{theoremalpha}
\label{thm:free-product-2cat-up}
Let $(\mathcal{C}_i)_{i\in I}$ be a family of rigid $C^*$-2-categories, and let $S$ be a set together with injections $f_i: S\to B(\mathcal{C}_i)$.
Then there exists a rigid $C^*$-2-category $\mathcal{C}$ together with unitary 2-embeddings $(\phi_i, \Phi_i):\mathcal{C}_i\to \mathcal{C}$ that satisfy $\phi_if_i=\phi_jf_j$ for all $i,j\in I$, and are universal in the following sense.

Given any rigid $C^*$-2-category $\mathcal{D}$ and unitary 2-functors $(\psi_i,\Psi_i):\mathcal{C}_i\to \mathcal{D}$ with the property that $\psi_if_i=\psi_jf_j$ for all $i,j\in I$, there exists a unitary 2-functor $(\psi,\Psi):\mathcal{C}\to\mathcal{D}$ such that $(\psi \phi_i,\Psi\Phi_i)\cong (\psi_i,\Psi_i)$ as unitary 2-functors.
If $(\psi',\Psi'):\mathcal{C}\to\mathcal{D}$ is another unitary 2-functor with the same properties, then $(\psi,\Psi)\cong (\psi',\Psi')$ as unitary 2-functors.
\end{theoremalpha}
The relevant terminology and notation will be introduced below.
Applied to rigid $C^*$-tensor categories in particular, the statement simplifies as follows.
\begin{theoremalpha}
\label{thm:free-product-up}
Let $(\mathcal{C}_i)_{i\in I}$ be a family of rigid $C^*$-tensor categories.
Then there exists a rigid $C^*$-tensor category $\mathcal{C}$ together with fully faithful unitary tensor functors $\Phi_i:\mathcal{C}_i\to \mathcal{C}$ that are universal in the following sense.

Given any rigid $C^*$-tensor category $\mathcal{D}$ and unitary tensor functors $\Psi_i:\mathcal{C}_i\to\mathcal{D}$, there exists a unitary tensor functor $\Psi:\mathcal{C}\to \mathcal{D}$ such that $\Psi\Phi_i\cong \Psi_i$ as unitary tensor functors.
If $\Psi':\mathcal{C}\to\mathcal{D}$ is another unitary tensor functor with the same property, then $\Psi\cong \Psi'$ as unitary tensor functors.
\end{theoremalpha}

The construction we present below is only one of several possible variations.
In the (formally simpler) case of rigid $C^*$-tensor categories, the rough idea is the following.
We know a posteriori that, given some family of rigid $C^*$-tensor categories $(\mathcal{C}_i)_{i\in I}$, the irreducible objects in the free product should be labelled by alternating words in the irreducibles of the factors.
This means that, as a $C^*$-category, the free product is essentially the same as the category $\mathcal{X}$ of finite-dimensional Hilbert spaces graded over these words.
One could of course attempt to define a tensor product on $\mathcal{X}$ directly, but this is notationally rather difficult.
Instead, we use the fact that $\mathcal{X}$ is a unitary $\mathcal{C}_i$-$\mathcal{C}_j$-bimodule category for all $i,j\in I$ in a canonical way.
The free product is then given by an appropriate category $\mathcal{C}$ of endofunctors of $\mathcal{X}$, that is in some sense generated by the copies of $\mathcal{C}_i$ acting from the left.
This way of working ensures that virtually all properties of $\mathcal{C}$ can be verified quite straightforwardly using induction arguments, essentially ``dealing with one letter at a time''.
The same idea applies (mutatis mutandis) to rigid $C^*$-2-categories.

\subsection*{Acknowledgement}

The author is indebted to Jonas Wahl for numerous fruitful discussions and valuable comments on earlier versions of this note.

\section{Definitions and notation}

\subsection{Rigid \texorpdfstring{$C^*$}{C*}-tensor 2-categories and unitary 2-functors}

The principal objects of study are rigid $C^*$-2-categories, as discussed in \cite{longo-roberts}.
Rigid $C^*$-2-categories are to rigid $C^*$-tensor categories as groupoids are to groups.
A rigid $C^*$-2-category $\mathcal{C}$ comes with a collection of 0-cells $B(\mathcal{C})$, which we always assume to be small (i.e.\@ a set).
For $a,b\in B(\mathcal{C})$, the category of 1-cells from $b$ to $a$ is a $C^*$-category denoted by $\mathcal{C}_{ab}$.
Morphisms in $\mathcal{C}_{ab}$ are called 2-cells, and the collection of 2-cells from $\alpha$ to $\alpha'$ in $\mathcal{C}_{ab}$ will be denoted by $(\alpha',\alpha)$.
The composition of $\alpha\in\mathcal{C}_{ab}$ and $\beta\in\mathcal{C}_{bd}$ is denoted by $\alpha\otimes\beta$ or simply $\alpha\beta$ if there is no danger of confusion.
By analogy with $C^*$-tensor categories, composition of 1-cells is only required to be associative up to a natural unitary isomorphism satisfying the usual pentagon relation.
Similarly, for all $a\in B(\mathcal{C})$, the category of 1-cells $\mathcal{C}_{aa}$ has a distinguished object $\epsilon_a$, acting as a local tensor unit.
We usually---except in the last section---assume that the tensor units $\epsilon_a$ are simple objects, and identify their endomorphism algebras with the complex numbers.
See \cite[\S~7]{longo-roberts} for further discussion.

The adjective ``rigid'' refers to the fact that every 1-cell $\alpha\in\mathcal{C}_{ab}$ has an essentially unique conjugate $\bar{\alpha}\in\mathcal{C}_{ba}$.
There additionally exist maps $s_\alpha:\epsilon_a\to\alpha\bar{\alpha}$ and $t_\alpha:\epsilon_b\to\bar{\alpha}\alpha$ such that
\[
    (s_\alpha^*\otimes\I)(\I\otimes t_\alpha) = \I
    \qquad\text{and}\qquad (t_\alpha^*\otimes\I)(\I\otimes s_\alpha) = \I\,.
\]
Such a pair is referred to as a \textit{solution to the conjugate equations} for $\alpha$.
If moreover $s_\alpha^*(T\otimes\I) s_\alpha=t_\alpha^*(\I\otimes T)t_\alpha$ for all endomorphisms $T$ of $\alpha$, then we say that the solution is \textit{standard}.
In this case, the functional $\Tr_\alpha(T)= s_\alpha^*(T\otimes\I) s_\alpha$ is tracial, independent of the choice of standard solution and referred to as the \textit{categorical trace} on $(\alpha,\alpha)$.
The quantity $\Tr_\alpha(\I)=s_\alpha^*s_\alpha=t_\alpha^*t_\alpha$ is called the \textit{quantum dimension} of $\alpha$, and denoted by $\qdim(\alpha)$

The existence of such conjugate objects forces $\mathcal{C}_{ab}$ to be semisimple for all $a,b\in B(\mathcal{C})$.
Moreover, all 2-cell spaces $(\alpha,\beta)$ for $\alpha,\beta\in\mathcal{C}_{ab}$ are finite-dimensional, and come equipped with a natural inner product given by
\[
    \langle T, S\rangle = \Tr_{\alpha}(TS^*) = \Tr_\beta(S^*T)
\]
for $T,S \in(\alpha,\beta)$.
In this way, all 2-cell spaces in $\mathcal{C}$ are finite-dimensional Hilbert spaces.
Throughout, we always assume that a full set of representatives of irreducible objects in $\mathcal{C}_{ab}$ has been chosen, and we denote it by $\Irr(\mathcal{C}_{ab})$.
We do not typically distinguish between classes of irreducible objects and their representatives.

Given $C^*$-2-categories $\mathcal{C}$ and $\mathcal{D}$, a \textit{unitary 2-functor} from $\mathcal{C}$ to $\mathcal{D}$ is a tuple $(f,F,F^{(2)},F^{(0)})$ consisting of 
\begin{itemize}
    \item a function $f:B(\mathcal{C})\to B(\mathcal{D})$;
    \item unitary functors $F_{ab}: \mathcal{C}_{ab}\to \mathcal{D}_{f(a)f(b)}$ for all $a,b\in B(\mathcal{C})$;
    \item natural unitaries $F^{(2)}_{\alpha,\beta}: F_{ab}(\alpha)F_{bc}(\beta)\to F_{ac}(\alpha\beta)$ for $a,b,c\in B(\mathcal{C})$, $\alpha\in\mathcal{C}_{ab}$, $\beta\in\mathcal{C}_{bc}$;
    \item unitaries $F^{(0)}_a:F(\epsilon_a)\to \epsilon_{f(a)}$ for $a\in B(\mathcal{C})$.
\end{itemize}
We often suppress the subscripts when they are clear from context.
Denoting the associators in $\mathcal{C}$ and $\mathcal{D}$ by $a$ and $a'$, respectively, we require
\[
    \begin{tikzcd}
    (F(\alpha)\otimes F(\beta))\otimes F(\gamma) \arrow[rr, "a_{F(\alpha),F(\beta),F(\gamma)}'"] \arrow[d, "F^{(2)}_{\alpha,\beta}\otimes\I_{F(\gamma)}"] & & F(\alpha)\otimes (F(\beta)\otimes F(\gamma)) \arrow[d, "\I_{F(\alpha)}\otimes F^{(2)}_{\beta,\gamma}"]\\
    F(\alpha\otimes\beta)\otimes F(\gamma)\arrow[d, "F^{(2)}_{\alpha\otimes\beta,\gamma}"] & & F(\alpha)\otimes F(\beta\otimes \gamma)\arrow[d,"F^{(2)}_{\alpha,\beta\otimes\gamma}"] \\
    F((\alpha\otimes\beta)\otimes\gamma) \arrow[rr, "F(a_{\alpha,\beta,\gamma})"] & & F(\alpha\otimes(\beta\otimes\gamma))
    \end{tikzcd}
\]
to commute for all composable 1-cells $\alpha,\beta,\gamma$ in $\mathcal{C}$, and similarly for $F^{(0)}$.
When $F^{(2)}$ and $F^{(0)}$ are equal to the identity everywhere, we say that $F$ is a \textit{strict} unitary 2-functor.

We usually simply refer to the entire tuple $(f,F,F^{(2)},F^{(0)})$ as $(f,F)$.
When the function $f$ is clear from context (especially when it is the identity map), we simply write $F$.
Given unitary 2-functors $(f,F)$ and $(f',F')$, we write $(f,F)\cong (f',F')$ whenever $f=f'$ and there exist unitary natural transformations
\[
    \eta: (F:\mathcal{C}_{ab}\to \mathcal{D}_{f(a)f(b)})\to (F':\mathcal{C}_{ab}\to \mathcal{D}_{f(a)f(b)})
\]
for all $a,b\in B(\mathcal{C})$, that are monoidal in the sense that $\eta F^{(2)}=(F')^{(2)}(\eta\otimes\eta)$.
We say that $(f,F)$ is a unitary \textit{2-embedding} if $f$ is injective and $F$ is fully faithful everywhere.
A sub-2-category $\mathcal{C}'\subset\mathcal{C}$ is \textit{2-full} if $\mathcal{C}'$ is closed under composition and conjugation of 1-cells (i.e.\@ tensor products) and $\mathcal{C}'_{ab}\subset\mathcal{C}_{ab}$ is a full, isomorphism-closed inclusion of $C^*$-categories for all $a,b\in B(\mathcal{C}')$---in particular, $\mathcal{C}'_{ab}$ is closed under direct sums and subobjects.

\subsection{Pointed \texorpdfstring{$C^*$}{C*}-2-categories}
\label{sec:pointed-cats}

Pointed $C^*$-2-categories are essentially the same as subfactor planar algebras and $\lambda$-lattices, considered from a more categorical point of view (see \cite[Remark~2.1]{arano-vaes}).
This formalism will be useful to establish the concrete link with the free composition of Bisch and Jones \cite{bisch-jones-interm-subfactors}, and with the results of Tarrago and Wahl on free wreath products \cite{tarrago-wahl}.
We revisit this point in the last section.
\begin{definition}
    A \textit{pointed} $C^*$\textit{-2-category} is a tuple $\mathcal{P}=(\mathcal{C},a,b,u)$, where 
    \begin{itemize}
    \item $\mathcal{C}$ is a rigid $C^*$-2-category with two distinct 0-cells given by $\mathcal{B}(\mathcal{C})=\{a,b\}$;
    \item $u$ is an object in $\mathcal{C}_{ab}$.
    \end{itemize}
    We say that $\mathcal{P}$ is \textit{nondegenerate} if every irreducible object in $\mathcal{C}_{aa}$ appears as a subobject of some tensor power of $u\bar{u}$.
\end{definition}
The natural morphisms between pointed 2-categories are unitary 2-functors preserving the distinguished object.
This can be made precise as follows.
\begin{definition}
    A \textit{morphism} of pointed $C^*$-2-categories from $(\mathcal{C},a,b,u)$ to $(\mathcal{C}',a',b',u')$ is a unitary 2-functor $(f,F):\mathcal{C}\to \mathcal{C}'$ such that $f(a)=a'$ and $f(b)=b'$, together with a unitary $p: F(u)\to u'$.
    Since the map $f$ is fixed and hence not part of the data, we suppress it.
    The composition of $(F,p)$ with another such morphism $(F',p'): (\mathcal{C}',a',b',u')\to (\mathcal{C}'',a'',b'',u'')$ is defined by $(F'F,p'F'(p))$.
    Hence, $(F,p)$ is an isomorphism if and only if $F$ is a unitary 2-equivalence.

    Given two such morphisms $(F_1,p_1):(\mathcal{C}_1,a_1,b_1,u_1)\to (\mathcal{C},a,b,u)$ and $(F_2, p_2):(\mathcal{C}_2,a_2,b_2,u_2)\to (\mathcal{C},a,b,u)$, we say that $F_1$ and $F_2$ are \textit{conjugate} if there exists an isomorphism $(G,p):(\mathcal{C}_1,a_1,b_1,u_1)\to (\mathcal{C}_2, a_2, b_2, u_2)$ together with an isomorphism of unitary 2-functors $\eta:F_2G \to F_1$ such that the diagram
    \[
        \begin{tikzcd}
        F_2G(u_1) \arrow[d, "\eta_{u_1}"] \arrow[r, "F_2(p)"] & F_2(u_2) \arrow[d, "p_2"]\\
        F_1(u_1)\arrow[r, "p_1"] & u
        \end{tikzcd}
    \]
    is commutative.
\end{definition}
We say that a morphism of pointed $C^*$-2-categories is \textit{dimension-preserving} if the underlying 2-functor is.
As explained in \cite[Remark~2.1]{arano-vaes}, a subfactor planar algebra encodes the same data as a generating standard simple $Q$-system in a rigid $C^*$-tensor category, which is in turn essentially the same as a nondegenerate pointed $C^*$-2-category.

\begin{example}
\label{exa:subfactor}
    Let $N\subset M$ be a finite-index inclusion of factors of type $\mathrm{II}_1$.
    One can associate a pointed $C^*$-2-category to $N\subset M$ as follows.
    Put $u=L^2(M)$, viewed as a Hilbert $N$-$M$-bimodule.
    Then define $\mathcal{C}_{N--N}$ as the category of Hilbert $N$-$N$-bimodules generated by taking subobjects and tensor powers of $u\bar{u}=u\mathbin{\overline{\otimes}_M} \bar{u}$.
    Note that $u\bar{u}\cong L^2(M)$ as Hilbert $N$-$N$-bimodules.
    Similarly, define $\mathcal{C}_{M--M}$ as the $C^*$-tensor category of Hilbert $M$-$M$-bimodules generated by $\bar{u}u=\bar{u}\mathbin{\overline{\otimes}_N} u$.
    If $N\subset M\subset M_1$ is the result of Jones' basic construction applied to $N\subset M$, then $\bar{u}u\cong L^2(M_1)$ as Hilbert $M$-$M$-bimodules.

    Finally, define $\mathcal{C}_{N--M}$ and $\mathcal{C}_{M--N}$ as the categories generated by taking subobjects of $u(\bar{u}u)^{\otimes n}$ (resp.\@ $\bar{u}(u\bar{u})^{\otimes n}$).
    Then $\mathcal{P}_{N\subset M}=(\mathcal{C}, N,M, u)$ is a pointed $C^*$-2-category.
    By Popa's reconstruction theorem from \cite{popa-lambda-lattices} and \cite[Remark~2.1]{arano-vaes}, every pointed $C^*$-2-category is isomorphic to $\mathcal{P}_{N\subset M}$ for some extremal finite-index subfactor $N\subset M$.
\end{example}

\begin{example}
\label{exa:pointed-tlj}
    The Temperley--Lieb--Jones (TLJ) planar algebra can be realised as a nondegenerate pointed $C^*$-2-category as follows.
    Fix $\delta\geq 2$ and $\tau=\pm\I$, and let $\mathcal{TLJ}_\delta$ be the Temperley--Lieb--Jones rigid $C^*$-tensor category with parameter $\delta$.
    That is to say, $\mathcal{TLJ}_{\delta}$ is generated by a self-conjugate irreducible object $u$, and there is a (standard) solution to the conjugate equations for $u$ given by a single morphism $R\in (uu,\epsilon)$ satisfying
    \[
        (R^*\otimes\I)(\I\otimes R) = \I
        \qquad R^*R = \qdim(u) = \delta\,.
    \]
    This intertwiner generates all morphisms between higher tensor powers of $u$, in the sense that any morphism from $u^{\otimes n}$ to $u^{\otimes m}$ can be written by composing morphisms of the form $\I^{\otimes i}\otimes R\otimes\I^{\otimes j}$ and their adjoints.
    Equivalently, the intertwiner spaces between higher tensor powers of $u$ are given by corners of the universal Temperley--Lieb $C^*$-algebra (see e.g.\@ \cite[\S~2.5]{neshveyev-tuset} for details).
    Now define $\mathcal{TLJ}_\delta^{\mathrm{even}}$ (resp.\@ $\mathcal{TLJ}_\delta^{\mathrm{odd}}$) as the full rigid $C^*$-tensor subcategory consisting of all objects in $\mathcal{TLJ}_\delta$ that embed into a direct sum of even (resp.\@ odd) tensor powers of $u$.

    Now put $\mathcal{C}_{aa}=\mathcal{C}_{bb}=\mathcal{TLJ}_\delta^{\mathrm{even}}$ and $\mathcal{C}_{ab}=\mathcal{C}_{ba}=\mathcal{TLJ}_\delta^{\mathrm{odd}}$.
    Then $(\mathcal{C},a,b,u)$ is a pointed $C^*$-2-category, which we refer to as the pointed TLJ 2-category with parameter $\delta$.
    These pointed $C^*$-2-categories are universal in a strong sense.
    Given any pointed rigid $C^*$-2-category $(\mathcal{D},a,b,v)$ with $\qdim(v)=\delta$, there exists an essentially unique morphism $\Phi$ from $(\mathcal{C},a,b,u)$ to $(\mathcal{D},a,b,v)$ that can be described as follows.

    Let $t_v: \epsilon_b\to \bar{v}v$ and $s_v: \epsilon_a\to v\bar{v}$ be a standard solution to the conjugate equations for $v$, and put $\Phi(u)=v$.
    The morphism $R:\epsilon\to uu$ is then mapped to $s_v$ when viewed as a 2-cell in $\mathcal{C}_{aa}$, and to $t_v$ when viewed as a 2-cell in $\mathcal{C}_{bb}$.
    The universal property of the Temperley--Lieb $C^*$-algebra \cite[see][p.~69]{neshveyev-tuset} ensures that this gives rise to a unitary 2-functor.
    See also \cite[Theorem~2.5.3]{neshveyev-tuset} for a similar statement about $\mathcal{TLJ}_\delta$ as a rigid $C^*$-tensor category.
\end{example}

\subsection{Modules over \texorpdfstring{$C^*$}{C*}-2-categories}

Modules over $C^*$-2-categories are defined in much the same way as module categories over tensor categories (see e.g.\@ \cite[Ch.~7]{egno15}).
Since we were unable to find the relevant definitions for 2-categories in the literature, we provide them here.
As a consequence of the fact that we work with categories equipped with an involution on the 2-cells, isomorphisms are typically required to be unitary.

\begin{definition}
    Let $B$ be a set.
    A $B$\textit{-graded} $C^*$\textit{-category} is a family of $C^*$-categories $(\mathcal{X}_b)_{b\in B}$.
\end{definition}

\begin{definition}
\label{def:two-cat-module}
    Let $\mathcal{C}$ be a $C^*$-2-category, and put $B=B(\mathcal{C})$.
    A \textit{unitary left} $\mathcal{C}$-\textit{module} consists of a $B$-graded $C^*$-category $\mathcal{X}$ and a collection of bi-unitary\footnote{We say that a bifunctor $F(-,-)$ is bi-unitary if it is bilinear and $F(\phi,\psi)^*=F(\phi^*,\psi^*)$ on morphisms $\phi,\psi$.} bifunctors $-\triangleright -:\mathcal{C}_{ab}\times\mathcal{X}_b\to \mathcal{X}_a$, together with unitary associators and unitors
    \[
        \mu_{\alpha,\beta,x}:\alpha\triangleright (\beta\triangleright x)\to (\alpha\otimes\beta)\triangleright x
        \qquad\text{and}\qquad
        \upsilon_x: \epsilon_c\triangleright x\to x
    \]
    for all $a,b,c\in B$ and $\alpha\in\mathcal{C}_{ab}$, $\beta\in\mathcal{C}_{bc}$ and $x\in\mathcal{X}_c$.
    For the associators, we require that
    \begin{equation}
    \label{eqn:left-module-assoc-diagram}
        \begin{tikzcd}
            & \alpha\triangleright(\beta\triangleright(\gamma\triangleright x))\arrow[ld, "\mu_{\alpha,\beta,\gamma\triangleright x}" description] \arrow[rd, "\I_\alpha\triangleright\mu_{\beta,\gamma,x}" description] &\\
            (\alpha\otimes\beta)\triangleright(\gamma\triangleright x) \arrow[d, "\mu_{x,\alpha,\beta\otimes\gamma}" description] & & \alpha\triangleright ((\beta\otimes\gamma)\triangleright x) \arrow[d, "\mu_{\alpha,\beta\otimes\gamma,x}" description]\\
            ((\alpha\otimes\beta)\otimes\gamma)\triangleright x \arrow[rr, "a_{\alpha,\beta,\gamma}\triangleright \I_x"] & & (\alpha\otimes(\beta\otimes\gamma))\triangleright x
        \end{tikzcd}
    \end{equation}
    commute for all $a,b,c,d\in B$ $\alpha\in\mathcal{C}_{ab}$, $\beta\in\mathcal{C}_{bc}$, $\gamma\in\mathcal{C}_{cd}$ and $x\in\mathcal{X}_d$.
    The unitors should satisfy
    \begin{equation}
    \label{eqn:left-module-unitor-diagram}
        \begin{tikzcd}
            (\alpha\triangleright \epsilon_b)\triangleright x\ar[rd, "\rho_\alpha\triangleright \I_x" description] \ar[rr, "\mu_{\alpha,\epsilon_b,x}"] & & \alpha\triangleright (\epsilon_b\triangleright x) \ar[ld, "\I_\alpha\triangleright \upsilon_x" description]\\
            & \alpha\triangleright x &
        \end{tikzcd}
    \end{equation}
    for all $a,b\in B$, $\alpha\in\mathcal{C}_{ab}$ and $x\in \mathcal{X}_b$.

    Analogously, a \textit{unitary right} $\mathcal{C}$-\textit{module} consists of a $B$-graded $C^*$-category $\mathcal{X}$ and a collection of bi-unitary bifunctors $-\triangleleft -:\mathcal{X}_{a}\times\mathcal{C}_{ab}\to\mathcal{X}_b$ together with unitary associators and unitors
    \[
        \mu_{x,\alpha,\beta}:(x\triangleleft \alpha)\triangleleft\beta\to x\triangleleft(\alpha\otimes\beta)
        \qquad\text{and}\qquad
        \upsilon_x: x\triangleleft \epsilon_a\to x
    \]
    for all $a,b,c\in B$, $\alpha\in\mathcal{C}_{ab}$, $\beta\in\mathcal{C}_{bc}$ and $x\in\mathcal{X}_c$.
    For the associators, we require that
    \begin{equation}
    \label{eqn:right-module-assoc-diagram}
        \begin{tikzcd}
            & ((x\triangleleft\alpha)\triangleleft\beta)\triangleleft\gamma\arrow[ld, "\mu_{x\triangleleft\alpha,\beta,\gamma}" description] \arrow[rd, "\mu_{x,\alpha,\beta}\triangleleft \I_\gamma" description] &\\
            (x\triangleleft \alpha)\triangleleft(\beta\otimes\gamma) \arrow[d, "\mu_{\alpha\otimes\beta,\gamma,x}" description] & & (x\triangleleft(\alpha\otimes\beta))\triangleleft\gamma \arrow[d, "\mu_{x,\alpha\otimes\beta,\gamma}" description]\\
            x\triangleleft (\alpha\otimes (\beta\otimes\gamma)) & & x\triangleleft ((\alpha\otimes \beta)\otimes\gamma) \arrow[ll, "\I_x\triangleleft a_{\alpha,\beta,\gamma}"]
        \end{tikzcd}
    \end{equation}
    commute for all $a,b,c,d\in B$ $\alpha\in\mathcal{C}_{ab}$, $\beta\in\mathcal{C}_{bc}$, $\gamma\in\mathcal{C}_{cd}$ and $x\in\mathcal{X}_a$. 
    The unitors should satisfy
    \begin{equation}
    \label{eqn:right-module-unitor-diagram}
        \begin{tikzcd}
            x\triangleleft(\epsilon_a\triangleleft\alpha)\arrow[rd, "\I_x\triangleleft\lambda_\alpha" description] \arrow[rr, "\mu_{x,\epsilon_a,\alpha,}"] & & (x\triangleleft \epsilon_a)\triangleleft \alpha \arrow[ld, "\upsilon_x\triangleleft \I_\alpha" description]\\
            & x\triangleleft\alpha &
        \end{tikzcd}
    \end{equation}
    for all $a,b\in B$, $\alpha\in\mathcal{C}_{ab}$ and $x\in \mathcal{X}_a$.
\end{definition}
For a $C^*$-2-category $\mathcal{C}$ and $c\in B(\mathcal{C})$, $\mathcal{C}_{-c}=(\mathcal{C}_{ac})_{a\in B(\mathcal{C})}$ (resp.\@ $\mathcal{C}_{c-}$) is a natural unitary left (resp.\@ right) $\mathcal{C}$-module.

\begin{remark}
    There is of course little reason why one could not consider modules coming from $B'$-graded $C^*$-categories together with a function $f:B(\mathcal{C})\to B'$.
    That being said, we state the definitions somewhat less generally to avoid making the notation more complicated than necessary.
    The ``restricted'' version with $B'=B(\mathcal{C})$ is already sufficient for our purposes.
\end{remark}

A $\mathcal{C}$-module also comes with a natural category of $\mathcal{C}$-linear functors.
\begin{definition}
\label{def:c-linear-category}
    Let $\mathcal{C}$ be a $C^*$-2-category, and put $B=B(\mathcal{C})$.
    Given unitary right $\mathcal{C}$-module categories $\mathcal{X}$ and $\mathcal{Y}$, define a $C^*$-category $\Hom_{-\mathcal{C}}(\mathcal{X},\mathcal{Y})$ as follows (compare \cite[\S~7.2]{egno15}):
    \begin{itemize}
        \item the objects are pairs $(F, c)$, where $F_b:\mathcal{X}_b\to \mathcal{X}_b$ is a unitary functor for all $b\in B$, and
        \[
            c_{x,\alpha}: F_a(x)\triangleleft\alpha \to F_b(x\triangleleft\alpha)
        \]
        is a unitary transformation, natural in both $\alpha\in\mathcal{C}_{ab}$ and $x\in\mathcal{X}_a$, that is compatible with the associators of $-\triangleleft -$.
        Concretely, the associator diagram
        \begin{equation}
        \label{eqn:right-commutant-assoc-diagram}
            \begin{tikzcd}
                & (F_a(x)\triangleleft \alpha)\triangleleft\beta\arrow[dl, "\mu_{F_a(x),\alpha,\beta}" description] \arrow[dr, "c_{x,\alpha}\triangleleft\I_\beta" description] & \\
                F_a(x)\triangleleft (\alpha\otimes\beta) \arrow[d, "c_{x,\alpha\otimes\beta}" description] & & F_b(x\triangleleft \alpha)\triangleleft \beta \arrow[d, "c_{x\triangleleft\alpha,\beta}" description] \\
                F_d(x\triangleleft (\alpha\otimes\beta)) & & F_d((x\triangleleft \alpha)\triangleleft\beta) \arrow[ll, "F_d(\mu_{x,\alpha,\beta})"]
            \end{tikzcd}
        \end{equation}
        should commute for all $a,b,d\in B$ and all $\alpha\in\mathcal{C}_{ab}$, $\beta\in\mathcal{C}_{bd}$.
        Additionally, the unitor diagram
        \begin{equation}
        \label{eqn:right-commutant-unitor-diagram}
            \begin{tikzcd}
                F_a(x)\triangleleft\epsilon_a\arrow[rr, "c_{x,\epsilon_a}"]\arrow[dr, "\upsilon_{F_a(x)}" description] & & F_a(x\triangleleft\epsilon_a) \arrow[dl, "F_a(\upsilon_a)" description]\\
                & F_a(x) &
            \end{tikzcd}
        \end{equation}
        should commute for all $a\in B$ and $x\in\mathcal{X}_a$.
        \item a morphism $(F,c)\to (F',c')$ is a natural map
        \[
            \eta_x: F_a(x)\to F_a'(x)
        \]
        for all $a\in B$ and $x\in\mathcal{X}_a$, such that 
        \begin{equation}
        \label{eqn:right-commutant-morphism-diagram}
            \begin{tikzcd}
                F_a(x)\triangleleft\alpha \arrow[d, "\eta_x\triangleleft\I_\alpha"]\arrow[r, "c_{x,\alpha}"] & F_b(x\triangleleft\alpha) \arrow[d, "\eta_{x\triangleleft\alpha}"]\\
                F_a'(x)\triangleleft\alpha \arrow[r, "c'_{x,\alpha}"]& F_b'(x\triangleleft\alpha)
            \end{tikzcd}
        \end{equation}
        commutes for all $a,b\in B$, $x\in\mathcal{X}_a$ and $\alpha\in\mathcal{C}_{ab}$.
    \end{itemize}
    The definition for left $\mathcal{C}$-modules is analogous.
\end{definition}
One can compose $(F,c):\Hom_{-\mathcal{C}}(\mathcal{Y},\mathcal{Z})$ and $(F',c'):\Hom_{-\mathcal{C}}(\mathcal{X},\mathcal{Y})$ into
\[
    (F,c)\tensorhat (F',c')=(FF',c\boxdot c') \in\Hom_{-\mathcal{C}}(\mathcal{X},\mathcal{Z})\,.
\]
Here $c\boxdot c'$ is given by
\begin{equation}
\label{eqn:commutant-morphism-composition}
    (c\boxdot c')_{x,\alpha}:
    F_a(F_a'(x))\triangleleft \alpha\to F_b(F_b'(x\triangleleft\alpha)):
    F_b(c_{x,\alpha}')c_{F_a'(x),\alpha}\,.
\end{equation}
for $\alpha\in\mathcal{C}_{ab}$ and $x\in\mathcal{X}_a$;
This composition is functorial: given morphisms $\eta:(F,c)\to (G,d)$ and $\zeta: (F',c')\to (G',d')$, setting
\[
    (\eta\tensorhat \zeta)_x: F_a(F_a'(x))\to G_a(G_a'(x)): \eta_{G_a'(x)}F_a(\zeta_x)
\]
for $x\in\mathcal{X}_a$ defines a morphism from $(FF',c\boxdot c')$ to $(GG',d\boxdot d')$.
Moreover, $\hat{\otimes}$ is strictly associative, and the identity on $\mathcal{X}$ for $\tensorhat$ is given by the pair $(\catid, \I_{-})$, where $\catid$ is the identity functor.
This allows us to view the category of right $\mathcal{C}$-modules as a (non-small, but strict) dagger 2-category.

\begin{remark}
    One can use this to prove that any $C^*$-2-category $\mathcal{C}$ is unitarily 2-equivalent to a strict one.
    For every $a\in B(\mathcal{C})$, we can view $\mathcal{C}_{a-}=(\mathcal{C}_{ab})_{b\in B}$ as a unitary right $\mathcal{C}$-module.
    The strict dagger 2-category $\mathcal{C}'$ over $B(\mathcal{C})$ given by $\mathcal{C}'_{ab}=\Hom_{-\mathcal{C}}(\mathcal{C}_{b-},\mathcal{C}_{a-})$ is unitarily 2-equivalent to $\mathcal{C}$.
    The proof is entirely analogous to \cite[Theorem~2.8.5]{egno15}.

    In light of this fact, we will often assume rigid $C^*$-2-categories to be strict when convenient.
\end{remark}

\section{Hilbert spaces graded over reduced words}
Let $(\mathcal{C}_i)_{i\in I}$ be a family of rigid $C^*$-2-categories, and suppose that we are given a set $S$ together with injections $f_i:S\to B(\mathcal{C}_i)$.
We denote the classes of 1-cells of $\mathcal{C}_i$ by $\mathcal{C}_{i|ab}$ for $a,b\in B(\mathcal{C}_i)$.
Put $B=\bigsqcup_i B(\mathcal{C}_i)/\sim$, where $\sim$ is the equivalence relation identifying the respective images of the $f_i$.
The map $\phi_i$ sending an element of $B(\mathcal{C}_i)$ to its equivalence class in $B$ is injective.
Extend $\mathcal{C}_i$ to $B$ by taking $\mathcal{C}_{i|aa}=\Hilb_f$ for $a\in B\setminus \phi_i(B(\mathcal{C}_i))$ and $\mathcal{C}_{i|ab}=0$ for $a\neq b$, $a,b\in B$ with at least one not in $\phi_i(B(\mathcal{C}_i))$.
For the purpose of constructing the free product, we may therefore assume that $B=B(\mathcal{C}_i)$ and that both $f_i$ and $\phi_i$ are the identity map for all $i\in I$ without significant loss of generality.

Suppose that we have fixed a set of representatives of irreducibles $\Irr(\mathcal{C}_{i|ab})$ for all $i\in I$, $a,b\in B$.
We denote the tensor unit of $\mathcal{C}_{i|aa}$ by $\epsilon_{i|a}$, and always assume that $\epsilon_{i|a}\in\Irr(\mathcal{C}_{i|aa})$.

A \textit{word} of type $(a,b)$ in $(\mathcal{C}_i)_{i\in I}$ is a sequence $w$ of symbols 
\[
    [\alpha_1]_{i_1}\cdots [\alpha_n]_{i_n}\,,
\]
such that $\alpha_k\in\mathcal{C}_{i_k|a_kb_k}$, $b_k=a_{k+1}$ for all $k$ and $a=a_1$, $b=b_n$.
For $w$ to be \textit{reduced}, we should additionally have that $i_k\neq i_{k+1}$ and $\alpha_k\in\Irr(\mathcal{C}_{i_{k|a_kb_k}})$.
Moreover, we require that $\alpha_k\neq \epsilon_{k|a_k}$ if $a_k=b_k$ in a reduced word.
The empty word is denoted by $\emptyword$.
By convention, we consider $\emptyword$ a reduced word of type $(a,a)$ for any $a\in B$.
For $i\in I$ and $a,b\in B$, let $W_{i|ab}^\ell$ be the set of reduced words of type $(a,b)$ that do not start with a letter in $\mathcal{C}_i$ (this includes the empty word when $a=b$).
Given $\alpha\in\Irr(\mathcal{C}_{i|ab})$ and $w\in W_{i|bc}^\ell$, we then define
\begin{equation}
\label{eqn:left-cons}
    \,[\alpha]_{i}:w = \begin{cases}
        w & a=b, \alpha = \epsilon_{i|a}\,\\
        [\alpha]_i w & \text{otherwise}
    \end{cases}
\end{equation}
For any $i\in I$, one can exhaust all reduced words by expressions of the form $[\alpha]_i:w$ for $\alpha\in\Irr(\mathcal{C}_{i|ab})$ and $w\in W_{i|bc}^\ell$.
The sets $W_{i|ab}^r$ and the notation $w:[\alpha]_i$ are defined analogously.

Let $W_{ab}$ be the set of all reduced words of type $(a,b)$, and denote by $\mathcal{X}_{ab}$ the category of finite-dimensional complex $W_{ab}$-graded Hilbert spaces $\mathcal{H}=\bigoplus_{w\in W_{ab}}\mathcal{H}_w$.
The morphisms in $\mathcal{X}_{ab}$ are linear maps respecting the grading.
That is to say, a morphism $\phi:\mathcal{H}\to\mathcal{K}$ in $\mathcal{X}_{ab}$ can be identified with a collection of linear maps $\phi_w:\mathcal{H}_w\to\mathcal{K}_w$ for $w\in W_{ab}$.
We denote the collection of $C^*$-categories obtained in this way by $\mathcal{X}$.

There is one distinguished object, denoted by $\star$ going forward, that is contained in $\mathcal{X}_{aa}$ for all $a\in B$.
It is defined by
\begin{equation}
\label{eqn:vacuum-object}
    \star_w = \begin{cases}
        \CC & w = \emptyword\,,\\
        0 & \text{otherwise}\,.
    \end{cases}
\end{equation}

\begin{lemma}
\label{thm:factor-actions}
For any $i\in I$ and $a,b,c\in B$, there is a bi-unitary bifunctor $-\triangleright_i -:\mathcal{C}_{i|ab}\times\mathcal{X}_{bc}\to \mathcal{X}_{ac}$ defined on 1-cells by
\[
    (\alpha\triangleright_i \mathcal{H})_{[\pi]_{i}:w}
    =\bigoplus_{\gamma\in\Irr(\mathcal{C}_{i|bd})} (\alpha\gamma,\pi) \otimes \mathcal{H}_{[\gamma]_i:w}
\]
for $d\in B$, $\alpha\in\mathcal{C}_{i|ab}$, $\mathcal{H}\in\mathcal{X}_{bc}$, $w\in W_{i|dc}^\ell$ and $\pi\in\Irr(\mathcal{C}_{i|ad})$.
We denote the embedding of $(\alpha\gamma,\pi) \otimes \mathcal{H}_{[\gamma]_i:w}$ into $(\alpha\triangleright_i \mathcal{H})_{[\pi]_{i}:w}$ by $(\delta_{\triangleright}^i)_{\pi,\gamma,w}^{\alpha,\mathcal{H}}$.

Given morphisms $\phi:\alpha\to\beta$ in $\mathcal{C}_{i|ab}$ and $T:\mathcal{H}\to\mathcal{K}$ in $\mathcal{X}_{bc}$, the associated morphism $\alpha\triangleright_i\mathcal{H}\to \beta\triangleright_i\mathcal{K}$ is given by
\[
    (\phi\triangleright_i T)\left((\delta_{\triangleright}^i)_{\pi,\gamma,w}^{\alpha,\mathcal{H}}(V\otimes\xi)\right)
    = (\delta_{\triangleright}^i)_{\pi,\gamma,w}^{\beta,\mathcal{K}}\left((\phi\otimes\I)V\otimes T\xi\right)
\]
for $d\in B$, $w\in W_{i|dc}^\ell$, $\pi\in\Irr(\mathcal{C}_{i|ad})$, $\gamma\in\Irr(\mathcal{C}_{i|bd})$, $V\in (\alpha\gamma,\pi)$ and $\xi\in\mathcal{H}_{[\gamma]_i:w}$.

These bifunctors turn $\mathcal{X}_{-c}=(\mathcal{X}_{ac})_{a\in B}$ into a unitary left $\mathcal{C}_i$-module for all $c\in B$.
\begin{proof}
    Fix $c\in B$ throughout the proof.
    We only need to specify unitors and associators for $-\triangleright_i-$, all other claims are trivial.

    Given $a,b,d\in B$, $\alpha\in \mathcal{C}_{i|ab}$, $\beta\in\mathcal{C}_{i|bd}$ and $\mathcal{H}\in\mathcal{X}_{dc}$, the associator
    \begin{equation}
    \label{eqn:left-module-associator}
        \mu_{\alpha,\beta,\mathcal{H}}^\ell: \alpha\triangleright_i(\beta\triangleright_i\mathcal{H}) \to \alpha\beta\triangleright_i \mathcal{H}
    \end{equation}
    is given by
    \[
        \mu_{\alpha,\beta,\mathcal{H}}^\ell\left(
            (\delta_{\triangleright}^i)_{\pi,\gamma,w}^{\alpha,\beta\triangleright_i \mathcal{H}}\left(
                V\otimes
                (\delta_{\triangleright}^i)_{\gamma,\gamma',w}^{\beta,\mathcal{H}}(W\otimes\xi)
            \right)
        \right)
        =\qdim(\gamma)^{1/2}(\delta_{\triangleright}^i)_{\pi,\gamma',w}^{\alpha\beta,\mathcal{H}}
        \left((\I\otimes W)V\otimes\xi\right)
    \]
    for all $e\in B$, $\pi\in\Irr(\mathcal{C}_{i|ae})$, $\gamma\in\Irr(\mathcal{C}_{i|be})$, $\gamma'\in\Irr(\mathcal{C}_{i|de})$, $w\in W_{i|ec}^\ell$, $V\in (\alpha\gamma,\pi)$, $W\in (\beta\gamma',\gamma)$ and $\xi\in\mathcal{H}_{[\gamma']_i:w}$.

    In the case where $a=b$, the unitor $\upsilon^\ell_{\mathcal{H}}:\epsilon_{i|a}\triangleright_i \mathcal{H}\to \mathcal{H}$ for $\mathcal{H}\in\mathcal{X}_{ac}$ is the obvious one given by the formula
    \begin{equation}
    \label{eqn:left-module-unitor}
        \upsilon^\ell_{\mathcal{H}}\left((\delta_\triangleright^i)^{\epsilon_{i|a},\mathcal{H}}_{\pi,\pi,w}(\I_\pi\otimes \xi)\right)=\xi
    \end{equation}
    for $e\in B$, $\pi\in \Irr(\mathcal{C}_{i|ae})$, $w\in W_{ec}^\ell$ and $\xi\in\mathcal{H}_{[\pi]_i:w}$.

    One easily checks that the coherence conditions \eqref{eqn:left-module-assoc-diagram} and \eqref{eqn:left-module-unitor-diagram} are satisfied.
\end{proof}
\end{lemma}

\begin{lemma}
\label{thm:factor-actions-right}
For any $i\in I$ and $a,b,c\in B$, there is a bi-unitary bifunctor $-\triangleleft_i -:\mathcal{X}_{ca}\times\mathcal{C}_{i|ab}\to \mathcal{X}_{cb}$ defined on 1-cells by
\[
    (\mathcal{H}\triangleleft_i \alpha)_{w:[\pi]_i}
    =\bigoplus_{\gamma\in\Irr(\mathcal{C}_{i|da})} \mathcal{H}_{w:[\gamma]_i}\otimes (\gamma\alpha,\pi)
\]
for $d\in B$, $\alpha\in\mathcal{C}_{i|ab}$, $\mathcal{H}\in\mathcal{X}_{ca}$, $w\in W_{i|cd}^r$ and $\pi\in\Irr(\mathcal{C}_{i|db})$.
We denote the embedding of $\mathcal{H}_{w:[\gamma]_i}\otimes (\gamma\alpha,\pi)$ into $(\mathcal{H}\triangleleft_i \alpha)_{w:[\pi]_i}$ by $(\delta_{\triangleleft}^i)_{w,\pi,\gamma}^{\alpha,\mathcal{H}}$.

Given morphisms $\phi:\alpha\to\beta$ in $\mathcal{C}_{i|ab}$ and $T:\mathcal{H}\to\mathcal{K}$ in $\mathcal{X}_{ca}$, the associated morphism $\mathcal{H}\triangleleft_i \alpha\to \mathcal{K}\triangleleft_i\beta$ is given by
\[
    (T\triangleleft_i \phi)\left((\delta_{\triangleleft}^i)_{w,\pi,\gamma}^{\mathcal{H},\alpha}(\xi\otimes V)\right)
    = (\delta_{\triangleleft}^i)_{w,\pi,\gamma}^{\mathcal{K},\beta}\left(T\xi\otimes (\I\otimes\phi)V\right)
\]
for $d\in B$, $w\in W_{i|cd}^r$, $\pi\in\Irr(\mathcal{C}_{i|db})$ and $\gamma\in\Irr(\mathcal{C}_{i|da})$, $V\in (\gamma\alpha,\pi)$ and $\xi\in\mathcal{H}_{w:[\gamma]_i}$

The bifunctors $-\triangleright_i-$ turn $\mathcal{X}_{c-}$ into a unitary right $\mathcal{C}_i$-module for all $c\in B$.
\begin{proof}
    Fix $c\in B$ throughout the proof.
    As in Lemma~\ref{thm:factor-actions}, we only need to specify unitors and associators for $-\triangleleft_i-$.

    Given $a,b,d\in B$, $\alpha\in \mathcal{C}_{i|ab}$, $\beta\in\mathcal{C}_{i|bd}$ and $\mathcal{H}\in\mathcal{X}_{ca}$, the associator
    \begin{equation}
    \label{eqn:right-module-associator}
        \mu_{\mathcal{H},\alpha,\beta}^r: (\mathcal{H}\triangleleft_i\alpha)\triangleleft_i \beta\to \mathcal{H}\triangleleft_i \alpha\beta
    \end{equation}
    is given by
    \[
        \mu_{\mathcal{H},\alpha,\beta}^r\left(
                (\delta_{\triangleleft}^i)_{w,\pi,\gamma}^{\mathcal{H}\triangleleft_i\alpha,\beta}\left(
                    (\delta_{\triangleleft}^i)_{w,\gamma,\gamma'}^{\mathcal{H},\alpha}(\xi\otimes W)
                    \otimes V
                \right)
        \right)
        =\qdim(\gamma)^{1/2}(\delta^i_\triangleleft)^{\mathcal{H},\alpha\beta}_{w,\pi,\gamma'}(\xi\otimes (W\otimes\I)V)
    \]
    for all $e\in B$, $\pi\in\Irr(\mathcal{C}_{i|ed})$, $\gamma\in\Irr(\mathcal{C}_{i|eb})$, $\gamma'\in\Irr(\mathcal{C}_{i|ea})$, $w\in W_{i|ce}^r$, $V\in (\gamma\beta,\pi)$, $W\in (\gamma'\alpha,\gamma)$ and $\xi\in\mathcal{H}_{w:[\gamma']_i}$.

    When $a=b$, the unitor $\upsilon^r_{\mathcal{H}}:\mathcal{H}\triangleleft_i \epsilon_{i|a}\to \mathcal{H}$ for $\mathcal{H}\in\mathcal{X}_{ca}$ is given by
    \begin{equation}
    \label{eqn:right-module-unitor}
        \upsilon^r_{\mathcal{H}}\left((\delta_\triangleleft^i)^{\mathcal{H},\epsilon_{i|a}}_{w,\pi,\pi}(\xi\otimes \I_\pi)\right)=\xi
    \end{equation}
    for $e\in B$, $\pi\in \Irr(\mathcal{C}_{i|ea})$, $w\in W_{ce}^r$ and $\xi\in\mathcal{H}_{w:[\pi]_i}$.
    One again easily verifies that these definitions satisfy the coherence conditions \eqref{eqn:right-module-assoc-diagram} and \eqref{eqn:right-module-unitor-diagram}.
\end{proof}
\end{lemma}

The left and right actions of $\mathcal{C}_i$ on $\mathcal{X}$ commute up to a canonical isomorphism.
In fact, the left $\mathcal{C}_i$-module structure commutes with \textit{any} of the right $\mathcal{C}_j$-actions in this way.
This is the content of the next lemma.
\begin{lemma}
\label{thm:left-right-bimodule}
    Fix $i,j\in I$, $a,b\in B$ and $\alpha\in\mathcal{C}_{i|ab}$.
    For all $c,d\in B$, $\alpha\in\mathcal{C}_{i|ab}$, $\beta\in\mathcal{C}_{j|cd}$ and $\mathcal{H}\in\mathcal{X}_{bc}$, there exists a natural unitary
    \[
        \Sigma^{ij}_{\alpha,\mathcal{H},\beta}:
        (\alpha\triangleright_i\mathcal{H})\triangleleft_j\beta
        \to \alpha\triangleright_i(\mathcal{H}\triangleleft_j\beta)
    \]
    such that $(\alpha\triangleright_i -, \Sigma^{ij}_{\alpha,-,-})\in\Hom_{-\mathcal{C}_j}(\mathcal{X}_{b-},\mathcal{X}_{a-})$.

    Additionally, for all $\alpha\in\mathcal{C}_{i|ab}$ and $\alpha'\in\mathcal{C}_{i|bc}$, the module associator maps defined in \eqref{eqn:left-module-associator} produce unitary morphisms
    \[
        \mu^\ell_{\alpha,\alpha',-}:
        (\alpha\triangleright_i -, \Sigma^{ij}_{\alpha,-,-})
        \tensorhat (\alpha'\triangleright_i -, \Sigma^{ij}_{\alpha',-,-})
        \to (\alpha\alpha'\triangleright_i -, \Sigma^{ij}_{\alpha\alpha',-,-})
    \]
    in $\Hom_{-\mathcal{C}_j}(\mathcal{X}_{c-},\mathcal{X}_{a-})$.
    For all $a\in B$, the unitors \eqref{eqn:left-module-unitor} also yield unitaries
    \[
        \upsilon^\ell_{-}: (\epsilon_{i|a}\triangleright_i-, \Sigma^{ij}_{\epsilon_{i|a},-,-})
        \to (\catid, \I_{-})
    \]
    in $\Hom_{-\mathcal{C}_j}(\mathcal{X}_{a-},\mathcal{X}_{a-})$.
\begin{proof}
    First, assume that $i\neq j$.
    Additionally choose $e,f\in B$, $\pi\in\Irr(\mathcal{C}_{i|ae})$, $\pi'\in\Irr(\mathcal{C}_{j|fd})$, $\gamma\in\Irr(\mathcal{C}_{i|be})$ and $\gamma'\in\Irr(\mathcal{C}_{j|fc})$, and $w\in W_{i|ef}^\ell \cap W_{j|ef}^r$.
    Then set
    \begin{equation}
    \label{eqn:left-right-bimodule-generic}
        \Sigma^{ij}_{\alpha,\mathcal{H},\beta}\left[
            (\delta_{\triangleleft}^j)^{\alpha\triangleright_i\mathcal{H},\beta}_{[\pi]_i:w,\pi',\gamma'}\left(
                (\delta_{\triangleright})^{\alpha,\mathcal{H}}_{\pi,\gamma,w:[\gamma']_j}(V\otimes \xi)\otimes V'
            \right)
        \right]
        =(\delta_{\triangleright}^i)^{\alpha,\mathcal{H}\triangleleft_j\beta}_{\pi,\gamma,w:[\pi']_j}\left(
            V\otimes (\delta^j_{\triangleleft})^{\mathcal{H},\beta}_{[\gamma]_i:w,\pi',\gamma'}(\xi\otimes V')\right)
    \end{equation}
    for $V\in(\alpha\gamma,\pi)$, $V'\in(\gamma'\beta,\pi')$ and $\xi\in\mathcal{H}_{[\gamma]_i:w:[\gamma']_j}$.

    In case $i=j$, the formula \eqref{eqn:left-right-bimodule-generic} is applicable as long as $w$ is not the empty word.
    To define $\Sigma^{ii}_{\alpha,\mathcal{H},\beta}$ on $((\alpha\triangleright_i\mathcal{H})\triangleleft_i\beta)_{\emptyword:[\pi]_i}$ for $\pi\in\Irr(\mathcal{C}_{i|ad})$, we need a different approach, since $\emptyword:[\pi]_i$ does not lie in $W_{i|ad}^\ell$ in this case.

    For $\pi\in\Irr(\mathcal{C}_{i|ad})$, $\gamma\in\Irr(\mathcal{C}_{i|bc})$, $\gamma'\in\Irr(\mathcal{C}_{i|ac})$, $V\in (\alpha\gamma,\gamma')$, $V'\in(\gamma'\beta,\pi)$ and $\xi\in\mathcal{H}_{[\gamma]_i}$ we instead put
    \begin{align*}
        &\Sigma^{ii}_{\alpha,\mathcal{H},\beta}\left[(\delta^i_{\triangleleft})^{\alpha\triangleright_i\mathcal{H},\beta}_{\emptyword,\pi,\gamma'}\numberthis\label{eqn:left-right-bimodule-edge}
            \left(
                (\delta^i_{\triangleright})^{\alpha,\mathcal{H}}_{\gamma',\gamma,\emptyword}(V\otimes\xi)\otimes V'
            \right)
        \right]\\
        &\quad =\qdim(\gamma')^{1/2}\sum_{\substack{\sigma\in\Irr(\mathcal{C}_{i|bd}) \\ W\in\onb(\gamma\beta,\sigma)}} \qdim(\sigma)^{1/2} (\delta^i_{\triangleright})^{\alpha,\mathcal{H}\triangleleft_i\beta}_{\pi,\sigma,\emptyword}\left(
            (\I\otimes W^*)(V\otimes\I)V'\otimes (\delta^i_{\triangleleft})^{\mathcal{H},\beta}_{\emptyword,\sigma,\gamma}(\xi\otimes W)
        \right)\,.
    \end{align*}
    An easy diagram chase shows that $\Sigma^{ij}_{\alpha,-,-}$ makes the diagrams \eqref{eqn:right-commutant-assoc-diagram} and \eqref{eqn:right-commutant-unitor-diagram} commute in either case.
    The commutativity of the diagram \eqref{eqn:right-commutant-morphism-diagram} for $\mu^\ell_{\alpha,-,-}$ follows from an analogous computation.
\end{proof}
\end{lemma}

The takeaway from Lemma~\ref{thm:left-right-bimodule} is that for any fixed $i\in I$, $-\triangleright_i -$ commutes with $-\triangleleft_j-$ for all $j\in I$.
Operations with this property can also be neatly formalised as a $C^*$-2-category $\mathcal{C}^\wr$.
For each $a,b\in B$, the objects in $\mathcal{C}^\wr_{ab}$ are pairs $(F, (c^j)_{j\in J})$ such that $(F,c_j)\in\Hom_{-\mathcal{C}_j}(\mathcal{X}_{b-},\mathcal{X}_{a-})$ for all $j\in I$ (see Definition~\ref{def:c-linear-category}).
The 2-cells in $\mathcal{C}^\wr$ are natural maps that satisfy the compatibility property for morphisms in $\Hom_{-\mathcal{C}_j}(\mathcal{X}_{b-},\mathcal{X}_{a-})$ for all $j\in I$, and the composition of 1-cells is defined exactly as in Definition~\ref{def:c-linear-category}.

The existence of a norm on $\mathcal{C}^\wr$ is guaranteed by the following lemma, which describes the 2-cells in $\mathcal{C}^\wr$ in terms of the distinguished object $\star$ defined in \eqref{eqn:vacuum-object}.
\begin{lemma}
\label{thm:c-wr-morphisms}
    Fix $a,b\in B$.
    For all 1-cells $(F,(c^j)_{j\in I})$, $(G,(d^j)_{j\in I})$ in $\mathcal{C}^\wr_{ab}$, the map that sends a 2-cell $\eta: (F,(c^j)_{j\in I})\to (G,(d^j)_{j\in I})$ to $\eta_\star: F(\star)\to G(\star)$ is $*$-preserving and injective.
    In particular, setting $\|\eta\|=\|\eta_\star\|$ provides a $C^*$-norm on $\mathcal{C}^\wr$.
\begin{proof}
    Fix a 2-cell $\eta: (F,(c^j)_{j\in I})\to (G,(d^j)_{j\in I})$.
    Choose $c\in B$ and a reduced word $w$ of type $(b,c)$.
    Write $w=w' [\pi]_i$.
    Define $\star\triangleleft\emptyword=\star$, and inductively set
    \[
        \star\triangleleft w = (\star\triangleleft w')\triangleleft_i \pi\,.
    \]
    The object $\star\triangleleft w$ is irreducible in $\mathcal{X}_{bc}$, and letting $w$ run over all reduced words of type $(b,c)$, we exhaust all irreducibles in $\mathcal{X}_{bc}$ up to isomorphism.
    Moreover, $\eta_{\star\triangleleft w}$ satisfies the relation
    \begin{equation}
    \label{eqn:c-wr-morphisms-unroll}
        \eta_{\star\triangleleft w}=d^i_{\star\triangleleft w',\pi}(\eta_{w'}\triangleleft_i \I_\pi) (c^i_{\star\triangleleft w',\pi})^*\,.
    \end{equation}
    In other words, $\eta_\star$ fully determines $\eta_{\star\triangleleft w}$ for all reduced words $w$.
    For general objects $\mathcal{H}\in\mathcal{X}_{bc}$, we then have that
    \begin{equation}
    \label{eqn:c-wr-morphisms-decompose}
        \eta_{\mathcal{H}}: F(\mathcal{H})\to G(\mathcal{H}):
        \eta_{\mathcal{H}}=\sum_{\substack{w\in W_{bc}\\ U\in\onb(\mathcal{H},\star\triangleleft w)}}
        G(U)\eta_{\star\triangleleft w}F(U)^*\,.
    \end{equation}
    The relations \eqref{eqn:c-wr-morphisms-unroll} and \eqref{eqn:c-wr-morphisms-decompose} imply that $\eta=0$ if and only if $\eta_\star=0$.
    This proves the lemma.
\end{proof}
\end{lemma}

The conclusion of this section can then be summarised as follows.
\begin{proposition}
\label{thm:c-wr-embedding}
    For every $i\in I$ and $a,b\in B$ the operation
    \[
        \Phi_i: \mathcal{C}_{i|ab}\to \mathcal{C}^\wr_{ab}:
        \alpha\mapsto \left(\alpha\triangleright_i-, (\Sigma^{ij}_{\alpha,-,-})_{j\in I}\right)
    \]
    defines a fully faithful unitary 2-functor from $\mathcal{C}_i$ to $\mathcal{C}^\wr$ as follows.

    For all $a,b\in B$, a 2-cell $\phi:\alpha\to\beta$ in $\mathcal{C}_{ab}$ is mapped to a 2-cell $F_i(\phi)$ in $\mathcal{C}_{ab}$ by
    \[
        \Phi_i: \Phi_i(\alpha)\to \Phi_i(\beta): \Phi_i(\phi)_{\mathcal{H}} = \phi\triangleright_i \I_{\mathcal{H}}
    \]
    for all $c\in B$ and $\mathcal{H}\in \mathcal{X}_{bc}$.
    The associator and unitor maps
    \[
        \Phi_i^{(2)}: \Phi_i(\alpha)\tensorhat \Phi_i(\beta)\to\Phi_i(\alpha\beta)\,,\quad
        \Phi_i^{(0)}: \Phi_i(\epsilon_{i|a})\to (\catid, (\I_{-})_{j\in J})
    \]
    are those defined in \eqref{eqn:left-module-associator} and \eqref{eqn:left-module-unitor}.
\begin{proof}
    The only part that still requires an argument is the assertion that $\Phi_i$ is fully faithful.
    
    The faithfulness is in fact automatic because $\Phi_i$ is a unitary 2-functor on a rigid $C^*$-2-category, but the direct proof is not difficult. 
    Fix $a,b\in B$ and $\alpha,\beta\in\mathcal{C}_{i|ab}$, and let $\phi:\alpha\to\beta$ be a 2-cell in $\mathcal{C}_{i|ab}$.
    If $\Phi_i(\phi)=0$, then certainly 
    \[
        0=\Phi_i(\phi)_\star\left((\delta_{\triangleright}^i)^{\alpha,\star}_{\pi,\epsilon_{i|b},\emptyword}(V\otimes 1)\right) = (\delta_{\triangleright}^i)_{\pi,\epsilon_{i|b},\emptyword}(\phi V\otimes 1)
    \]
    and hence $\phi V=0$ for $\pi\in\Irr(\mathcal{C}_{ab})$ and $V\in (\alpha,\pi)$.
    This means that $\phi=0$.

    By faithfulness and Lemma~\ref{thm:c-wr-morphisms}, we moreover have that
    \begin{align*}
        \dim_\CC (\alpha,\beta)
        &\leq \dim_\CC (\Phi_i(\alpha),\Phi_i(\beta))
        \leq \dim_\CC (\alpha\triangleright\star, \beta\triangleright\star)\\
        &=\sum_{\pi\in\Irr(\mathcal{C}_{i|ab})} \mult(\pi,\alpha)\mult(\pi,\beta)
        =\dim_\CC(\alpha,\beta)
    \end{align*}
    so all inequalities are equalities, implying that $\Phi_i$ is also full.
\end{proof}
\end{proposition}
We can now construct the free product inside\footnote{In fact, the subcategory of $\mathcal{C}^\wr$ that we are about to construct is monoidally equivalent to $\mathcal{C}^\wr$, but the specific form we use below is somewhat more convenient when proving the universal property.} $\mathcal{C}^\wr$; this is the content of the next section.

\section{The free product 2-category}
We preserve the notation used in the previous section.
Lemma~\ref{thm:factor-actions} allows us to view objects in any of the $\mathcal{C}_i$'s as objects in a $C^*$-2-category of functors $\mathcal{C}^\wr$.
We claim that this generates a $C^*$-2-category $\mathcal{C}$ that satisfies the universal property of Theorem~\ref{thm:free-product-2cat-up}.

More precisely, given $a,b\in B$, the 1-cells in $\mathcal{C}_{ab}$ are finite formal sums $\bigoplus_k w_k$, where $w_k$ is an arbitrary (i.e.\@ possibly non-reduced) word of type $(a,b)$ in $(\mathcal{C}_i)_{i\in I}$.
To define morphisms between 1-cells in $\mathcal{C}$, we first have to realise them as 1-cells in $\mathcal{C}^\wr$.
The following recursive definition produces an object $(w\triangleright -, (c^j_w)_{j\in I})$ in $\mathcal{C}^\wr_{ab}$ for all words $w$ of type $(a,b)$.
\begin{itemize}
\item If $a=b$ and $w=\emptyword$, then $\emptyword\triangleright -$ is the unit object in $\mathcal{C}^\wr_{aa}$.
\item If $w=[\alpha]_i w'$ for $\alpha\in\mathcal{C}_{i|ad}$ and $w'$ a word of type $(d,b)$, then
\begin{align*}
    w\triangleright - &= \alpha\triangleright_i (w'\triangleright - )\,,\qquad\text{and}\\
    c^j_w &= \Sigma^{ij}_{\alpha,-,-}\boxdot c^j_{w'}\numberthis\label{eqn:cwr-composition}
\end{align*}
for all $j\in I$, where the $\boxdot$-operation is taken as in \eqref{eqn:commutant-morphism-composition}.
\end{itemize}
This definition extends to sums of words in the obvious way---we therefore usually work with individual words instead of sums.
We denote the object in $\mathcal{C}^\wr_{ab}$ associated with $w\in\mathcal{C}_{ab}$ by $\mathcal{L}_w$.

Given $a,b\in B$ and $w,w'\in\mathcal{C}_{ab}$, the morphisms from $w$ to $w'$ are by definition morphisms from $\mathcal{L}_w$ to $\mathcal{L}_{w'}$ in $\mathcal{C}^\wr_{ab}$.
In particular, Lemma~\ref{thm:c-wr-morphisms} provides us with a $C^*$-norm on 2-cells in $\mathcal{C}$.
The composition of two words $w$ of type $(a,b)$ and $w'$ of type $(b,c)$ is the word of type $(a,c)$ given by concatenation, i.e.\@ $ww'$.
By definition, we have that
\[
    \mathcal{L}_w\tensorhat \mathcal{L}_{w'} = \mathcal{L}_{ww'}
\]
with strict equality.
It follows that the composition of 1-cells in $\mathcal{C}$ inherits its functorial structure from $\mathcal{C}^\wr$.

Before proving Theorem~\ref{thm:free-product-2cat-up}, we can already show that $\mathcal{C}$ satisfies the usual ``working definition'' of a free product.
\begin{proposition}
\label{thm:poor-mans-free-product}
    The assignment $\Phi_i: \mathcal{C}_{i|ab}\to \mathcal{C}_{ab}: \alpha\mapsto [\alpha]_i$, $a,b\in B$ canonically defines a fully faithful unitary 2-functor from $\mathcal{C}_i$ to $\mathcal{C}$.

    For any reduced word $w$ of type $(a,b)$, the corresponding object $w\in\mathcal{C}_{ab}$ is irreducible.
    Moreover, these irreducible objects are pairwise non-isomorphic, and exhaust all irreducible objects in $\mathcal{C}_{ab}$ up to isomorphism.
\begin{proof}
    The first part is a rewording of Proposition~\ref{thm:c-wr-embedding}, and we preserve the notation $\Phi_i$ used there.
    The second claim follows from Lemma~\ref{thm:c-wr-morphisms}.
    Indeed, it is immediate from Lemma~\ref{thm:c-wr-morphisms} that an object $(F, (c^j)_{j\in I})$ in $\mathcal{C}^\wr_{ab}$ is irreducible if $F(\star)\in\mathcal{X}_{ab}$ is one-dimensional, where $\star$ is the distinguished object in $\mathcal{X}_{bb}$.
    By induction on the length of $w$, it is easy to see that $w\triangleright \star$ is one-dimensional for all reduced words $w$.
    If $w'$ is another reduced word of type $(a,b)$, then there are no morphisms from $w\triangleright\star$ to $w'\triangleright\star$ in $\mathcal{X}_{ab}$ unless $w=w'$, so $w\not\cong w'$ in $\mathcal{C}$ whenever $w\neq w'$.

    To prove that these exhaust all irreducibles, let $v$ be a general word of type $(a,b)$ such that $v$ is an irreducible object in $\mathcal{C}_{ab}$.
    We argue by induction on the length that $v$ must be (isomorphic to) a reduced word of length less than or equal to that of $v$.

    When $v$ is the empty word, there is nothing to prove.
    Moreover, when $v=[\alpha]_i$ for $\alpha\in\mathcal{C}_{i|ab}$, the functoriality of $\Phi_i$ guarantees that $\alpha$ must be irreducible, which clearly implies the claim.
    For a word $v$ of length $n\geq 2$, we can write $v=[\alpha]_i [\beta]_j v'$, where $\alpha\in\mathcal{C}_{i|ac}$, $\beta\in\mathcal{C}_{j|cd}$ and $v'$ is a word of type $(d,b)$.
    A composition of two $1$-cells in a $C^*$-2-category can only be irreducible if both factors are.
    By appealing to the induction hypothesis, we therefore know that $[\alpha]_i$ and $[\beta]_jv'$ are isomorphic to reduced words $w$ and $w'$ of type $(a,c)$ and $(c,d)$, respectively.
    Hence, $v$ is isomorphic to the composition of $w$ and $w'$.

    If either of $w$ or $w'$ are empty, then we are done, so assume that $w$ and $w'$ are both non-empty words.
    This implies that $w=[\pi]_i$ for some $\pi\in\mathcal{C}_{i|ac}$.
    If $w'=[\pi']_iw''$ with $\pi'\in\Irr(\mathcal{C}_{i|ce})$ and $w''\in W_{i|ed}^\ell$, then the fact that $\Phi_i$ is a 2-functor implies that
    \[
        v\cong [\pi\pi']_i w''.
    \]
    Since $w''$ has length at most $n-2$, we can apply the induction hypothesis to arrive at the conclusion.
    Finally, if the first letter of $w'$ is \textit{not} from $\mathcal{C}_i$, then $ww'$ is itself reduced and isomorphic to $v$, so we are done.
\end{proof}
\end{proposition}

Note that $\mathcal{C}$ is rigid, since $[\bar{\alpha}]_i$ is a conjugate for $[\alpha]_i$ in $\mathcal{C}$ for all 1-cells $\alpha$ in $\mathcal{C}_i$, which extends to all 1-cells in $\mathcal{C}$ by taking tensor products and direct sums.
In particular, the quantum dimension of a word $v=[\alpha_1]_{i_1}\cdots [\alpha_n]_{i_n}$ in $\mathcal{C}$ is given by $\qdim(v)=\qdim(\alpha_1)\cdots\qdim(\alpha_n)$.

To upgrade Proposition~\ref{thm:poor-mans-free-product} into the universal property required by Theorem~\ref{thm:free-product-2cat-up}, some more effort is required.
The main technical ingredient in the proof is the ``assembly map'' we construct below.

\begin{lemma}
\label{thm:assembly-map}
    Suppose that $\mathcal{D}$ is a strict $C^*$-2-category and $(g,\Psi_i):\mathcal{C}_i\to\mathcal{D}$ a unitary 2-functor for all $i\in I$.
    Given $a,b\in B$ and a word $v=[\alpha_1]_{i_1}\cdots [\alpha_n]_{i_n}$ of type $(a,b)$, define an object $\Psi(v)$ in $\mathcal{D}_{ab}$ by
    \[
        \Psi(v)=\Psi_{i_1}(\alpha_1)\cdots \Psi_{i_n}(\alpha_n)\,.
    \]
    For all $a,b\in B$, all words $v$ of type $(a,b)$ and all $w\in W_{ab}$, there exists a linear map
    \[
        \Psi_{v,w}: (v\triangleright\star)_w \to (\Psi(v), \Psi(w))
    \]
    satisfying the following relations.
    \begin{itemize}
        \item For all words $v,w,w'$ of type $(a,b)$ with $w,w'$ reduced and all $\zeta\in (v\triangleright \star)_w$, $\zeta'\in\ (v\triangleright\star)_{w'}$, we have that
        \begin{equation}
        \label{eqn:assembly-map-orthogonality}
            (\Psi_{v,w}\zeta)^* (\Psi_{v,w'}\zeta')
            = \begin{cases}
               \qdim(w)^{-1}\langle\zeta',\zeta\rangle\I_{\Psi(w)} & w=w'\,,\\
               0 & \text{otherwise}\,.
            \end{cases}
        \end{equation}
        in $(\Psi(w),\Psi(w'))$.
        \item For $a,b,c,d\in B$, $\alpha\in\mathcal{C}_{i|ab}$, words $v,v'$ of type $(b,c)$, $w\in W_{i|dc}^\ell$, $\gamma\in\Irr(\mathcal{C}_{i|bd})$, $\xi\in (v\triangleright\star)_{[\gamma]_i:w}$ and $\xi'\in (v'\triangleright\star)_{[\gamma]_i:w}$, the identity
        \begin{align*} 
            \sum_{\substack{\pi\in\Irr(\mathcal{C}_{i|ad}) \\ V\in\onb(\alpha\gamma,\pi)}} 
            &\qdim(\pi) \Psi_{[\alpha]_i v,[\pi]_i:w}\left((\delta_{\triangleright}^i)^{\alpha,v\triangleright\star}_{\pi,\gamma,w}(V\otimes\xi)\right)
            \Psi_{[\alpha]_i v',[\pi]_i:w} \left((\delta_{\triangleright}^i)^{\alpha,v'\triangleright\star}_{\pi,\gamma,w}(V\otimes\xi')\right)^*\\
            &=\qdim(\gamma)\I_{\Psi_i(\alpha)} \otimes \left((\Psi_{v,[\gamma]_i:w}\xi)(\Psi_{v',[\gamma]_i:w}\xi')^*\right)\,,\numberthis\label{eqn:assembly-map-left-identity}
        \end{align*}
        holds in $(\Psi_i(\alpha)\Psi(v),\Psi_i(\alpha)\Psi(v'))$.
        \item
        Choose $a,b,c\in B$, $\alpha\in\mathcal{C}_{i|ab}$, and let $T_\alpha$ be the canonical unitary isomorphism from $\star\triangleleft_i\alpha$ to $\alpha\triangleright_i\star$.
        For any word $u$ of type $(c,a)$, we then define a natural unitary
        \begin{equation}
        \label{eqn:sigma-rearrange}
            \tilde{\Sigma}_{u,\alpha}:
            (u\triangleright\star)\triangleleft_i \alpha\to u[\alpha]_i\triangleright\star:
            \tilde{\Sigma}_{u,\alpha}=(u\triangleright T_\alpha) (c^i_u)_{\star,\alpha}\,;
        \end{equation}
        see also \eqref{eqn:cwr-composition}.  
        For $d\in B$, words $v,v'$ of type $(c,a)$, $w\in W_{i|cd}^r$, $\gamma\in\Irr(\mathcal{C}_{i|da})$, $\xi\in(v\triangleright\star)_{w:[\gamma]_i}$ and $\xi'\in(v'\triangleright\star)_{[\gamma]_i:w}$, the right-handed version of \eqref{eqn:assembly-map-left-identity} reads as
        \begin{align*}
            \sum_{\substack{\pi\in\Irr(\mathcal{C}_{i|db}) \\ V\in\onb(\gamma\alpha,\pi)}} 
            &\qdim(\pi) \Psi_{v[\alpha]_i,w:[\pi]_i}\left(\tilde{\Sigma}_{v,\alpha}(\delta_{\triangleleft}^i)^{v\triangleright\star,\alpha}_{w,\pi,\gamma}(\xi\otimes V)\right)
            \Psi_{v'[\alpha]_i,w:[\pi]_i}\left(\tilde{\Sigma}_{v',\alpha}(\delta_{\triangleleft}^i)^{v\triangleright\star,\alpha}_{w,\pi,\gamma}(\xi'\otimes V)\right)^*\\
            &=\qdim(\gamma)\left((\Psi_{v,w:[\gamma]_i}\xi)(\Psi_{v',w:[\gamma]_i}\xi')^*\right)\otimes\I_{\Psi_i(\alpha)}\,.\numberthis\label{eqn:assembly-map-right-identity}
        \end{align*}
    \end{itemize}
\begin{proof}
    Denote the unit in $\mathcal{D}_{g(a)g(a)}$ by $\epsilon_{g(a)}$.
    Throughout the proof, we assume that $\Psi^{(0)}_i:\Psi_i(\epsilon_{i|a})\to \epsilon_{g(a)}$ is the identity map everywhere, for notational simplicity.
    Accounting for it explicitly is straightforward, so the lemma remains true in general.

    In order to define $\Psi_{v,w}$, we proceed by induction on the length of $v$.
    For $a=b$ and $v=\emptyword$, we map $1\in\star_{\emptyword}$ to the identity on the unit in $\mathcal{D}_{aa}$ if $w=\emptyword$, and use the zero map otherwise.

    Now take $v=[\alpha]_iv'$ with $c\in B$, $\alpha\in\mathcal{C}_{i|ac}$, and $v'$ a word of type $(c,b)$.
    Write $w=[\pi]_i:w'$ with $d\in B$, $w\in W_{i|db}^\ell$ and $\pi\in\Irr(\mathcal{C}_{i|ad})$, and put
    \begin{align*}
        \Psi_{v,w}: &(\alpha\triangleright_i (v'\triangleright \star))_w\to (\Psi_i(\alpha)\Psi(v'),\Psi_i(\pi)\Psi(w')):\\
        &\ (\delta_{\triangleright}^i)_{\pi,\gamma,w'}^{\alpha,v'\triangleright\star}(V\otimes\xi)
        \mapsto\qdim(\gamma)^{1/2}(\I_{\Psi_i(\alpha)}\otimes \Psi_{v',[\gamma]_i:w'}\xi)\left((\Psi^{(2)}_i)^*(\Psi_i(V))\otimes\I_{\Psi(w')}\right)\,.
    \end{align*}
    Here, $\gamma\in\Irr(\mathcal{C}_{i|bd})$, $V\in (\alpha\gamma,\pi)$ and $\xi\in (v'\triangleright\star)_{[\gamma]_i:w}$.

    The relation \eqref{eqn:assembly-map-orthogonality} can also be proven by induction on the length of $v$.
    For $v=\emptyword$, the statement is clear.
    Write $w=[\pi]_i:u$, $w'=[\pi']_i:u'$ for $d,d'\in B$, $u\in W_{i|db}^\ell$, $u'\in W_{i|d'b}^\ell$, $\pi\in\Irr(\mathcal{C}_{i|ad})$ and $\pi'\in\Irr(\mathcal{C}_{i|ad'})$.

    \noindent%
    Now choose $\gamma\in\Irr(\mathcal{C}_{i|bd})$, $\gamma'\in\Irr(\mathcal{C}_{i|bd'})$, $\xi\in (v'\triangleright\star)_{[\gamma]_i:u}$, $\xi'\in (v'\triangleright\star)_{[\gamma']_i:u'}$, $V\in (\alpha\gamma,\pi)$ and $V'\in(\alpha\gamma',\pi')$.
    To check \eqref{eqn:assembly-map-orthogonality} for the vectors
    \[
        \zeta=(\delta^i_{\triangleright})^{\alpha,v'\triangleright\star}_{\pi,\gamma,u}(V\otimes\xi)
        \qquad\text{and}\qquad
        \zeta'=(\delta^i_{\triangleright})^{\alpha,v'\triangleright\star}_{\pi',\gamma',u'}(V'\otimes\xi')\,
    \]
    we make the following computation in $(\Psi_i(\pi)\Psi(u),\Psi_i(\pi')\Psi(u'))$ based on Schur's lemma and the induction hypothesis:
    \begin{align*}
        \qdim(&\gamma)^{1/2}\qdim(\gamma')^{1/2}
           \left(\Psi_i(V^*)\Psi^{(2)}_i\otimes\I_{\Psi(u)}\right) (\I_{\Psi_i(\alpha)}\otimes (\Psi_{v',[\gamma]_i:u}\xi)^*(\Psi_{v',[\gamma']_i:u'}\xi'))\\
        &\qquad\qquad\qquad\qquad\qquad \left((\Psi^{(2)}_i)^*(\Psi_i(V'))\otimes\I_{\Psi(u')}\right)\\
        &=\delta_{\gamma,\gamma'}\delta_{u,u'} \qdim(u)^{-1}\langle \xi',\xi\rangle \left(\Psi_i(V^*)\Psi^{(2)}_i(\Psi^{(2)}_i)^*(\Psi_i(V'))\otimes\I_{\Psi(u)}\right)\\
        &=\delta_{\gamma,\gamma'} \qdim(u)^{-1}\langle \xi',\xi\rangle \Psi_i(V^*V')\otimes\I_{\Psi(u)}\\
        &=\qdim(u)^{-1}\qdim(\pi)^{-1}\delta_{\gamma,\gamma'}\delta_{u,u'}\delta_{\pi,\pi'} \langle \xi',\xi\rangle \Tr_{\pi}(V^*V')\I_{\Psi_i(\pi)\Psi(u)}\\
        &=\qdim(w)^{-1}\delta_{w,w'}\langle (\delta^i_{\triangleright})^{\alpha,v'\triangleright\star}_{\pi',\gamma',u'}(V'\otimes\xi'), (\delta^i_{\triangleright})^{\alpha,v'\triangleright\star}_{\pi,\gamma,u}(V\otimes\xi)\rangle\I_{\Psi_i(\pi)\Psi(u)}\,.
    \end{align*}
    This proves \eqref{eqn:assembly-map-orthogonality}.
    The identity \eqref{eqn:assembly-map-left-identity} follows immediately from the definitions.
    Similarly, \eqref{eqn:assembly-map-right-identity} is an immediate consequence of the fact that
    \begin{equation}
    \label{eqn:assembly-map-right-version}
        \Psi_{v[\alpha]_i,w:[\pi]_i}\left(\tilde{\Sigma}_{v,\alpha}(\delta_{\triangleleft}^i)^{v\triangleright\star,\alpha}_{w,\pi,\gamma}(\xi\otimes V)\right)
        =\qdim(\gamma)^{1/2} (\Psi_{v,w:[\gamma]_i}\xi,\I_{\Psi_i(\alpha)})\left(\I_{\Psi(w)}\otimes(\Psi_i^{(2)})^*(\Psi_i(V))\right)
    \end{equation}
    holds in $(\Psi(v)\Psi_i(\alpha),\Psi(w)\Psi_i(\pi))$
    for $a,b,c,d\in B$, $\alpha\in\mathcal{C}_{i|ab}$, any word $v$ of type $(c,a)$, $w\in W_{i|cd}$, $\pi\in\Irr(\mathcal{C}_{i|db})$, $\gamma\in\Irr(\mathcal{C}_{i|da})$, $\xi\in (v\triangleright\star)_{w:[\pi]_i}$ and $V\in (\alpha\gamma,\pi)$.
    The identity \eqref{eqn:assembly-map-right-version} clearly holds when $v=\emptyword$.
    Proceeding by induction on the length of $v$, write $v=[\beta]_jv'$.
    Using the fact that $c_v^i=\Sigma_{\beta,-,-}^{ji}\boxdot c^i_{v'}$ together with the explicit expressions for $\Sigma_{\beta,-,-}^{ji}$ given in \eqref{eqn:left-right-bimodule-generic} and \eqref{eqn:left-right-bimodule-edge} one readily obtains \eqref{eqn:assembly-map-right-version} after a straightforward (albeit somewhat tedious) computation.
\end{proof}
\end{lemma}

We now prove that the $C^*$-2-category $\mathcal{C}$ we constructed in this section satisfies the universal property of Theorem~\ref{thm:free-product-2cat-up}.
The notation $\Psi(v)$ in the statement of the above lemma is suggestive; this is exactly how we will define the lift of the $\Psi_i$'s to $\mathcal{C}$.

\begin{proof}[Proof of Theorem~\ref{thm:free-product-2cat-up}]
    The unitary 2-functors $\Phi_i:\mathcal{C}_i\to\mathcal{C}$ are provided by Proposition~\ref{thm:poor-mans-free-product}.

    Suppose now that we are given a strict $C^*$-2-category $\mathcal{D}$ with unitary 2-functors $(g,\Psi_i):\mathcal{C}_i\to\mathcal{D}$.
    In this setup, we can even write down a strict 2-functor $(g,\Psi)$ from $\mathcal{C}$ to $\mathcal{D}$ that has the desired properties.
    On 1-cells, it is clear that we should define
    \[
        \Psi([\alpha_1]_{i_1}\cdots [\alpha_n]_{i_n})
        =\Psi_{i_1}(\alpha_1)\cdots \Psi_{i_n}(\alpha_n)\,.
    \]
    The empty word in $\mathcal{C}_{aa}$ is mapped to the tensor unit in $\mathcal{D}_{i|g(a)g(a)}$; in particular, $\Psi^{(0)}$ is the identity functor.

    The action of $\Psi$ on 2-cells is more involved, but Lemma~\ref{thm:assembly-map} does most of the heavy lifting at this point.
    For a 2-cell $\eta:v\to v'$ between words of type $(a,b)$, the corresponding map $\Psi(\eta)$ in $(\Psi(v'),\Psi(v))$ is given by
    \begin{equation}
    \label{eqn:universal-functor-morphism}
        \Psi(\eta) = \sum_{w\in W_{ab}} \sum_{\xi\in\onb((v\triangleright\star)_w)} \qdim(w) \Psi_{v',w}(\eta_\star\xi) (\Psi_{v,w}\xi)^*\,,
    \end{equation}
    where $\Psi_{v,w}$ is the assembly map defined in Lemma~\ref{thm:assembly-map}.

    The relation \eqref{eqn:assembly-map-orthogonality} ensures that \eqref{eqn:universal-functor-morphism} respects composition and involution of 2-cells.
    To verify that $\Psi$ preserves the identity and produces a strict 2-functor, we still have to check that
    \begin{equation}
    \label{eqn:universal-functor-tensor-compat}
        \Psi(\I_v\tensorhat\eta) = \I_{\Psi(v)}\otimes\Psi(\eta)
        \qquad\text{and}\qquad
        \Psi(\eta\tensorhat\I_{v'}) = \Psi(\eta)\otimes\I_{\Psi(v')}
    \end{equation}
    for all $a,b,c\in B$, all words $v_1, v_2$ of type $(c,a)$, 2-cells $\eta: v_1\to v_2$ and all words $v$ (resp.\@ $v'$) of type $(b,c)$ (resp.\@ $(a,b)$).
    Since the tensor products in $\mathcal{C}$ and $\mathcal{D}$ are strictly associative, it suffices to do this for words consisting of a single letter.
    The left identity in \eqref{eqn:universal-functor-tensor-compat} is then an immediate consequence of \eqref{eqn:assembly-map-left-identity}---this also implies that $\Psi$ preserves the identity map on all objects.

    The right identity in \eqref{eqn:universal-functor-tensor-compat} uses the fact that $\eta$ commutes with the right $\mathcal{C}_i$-module structures.
    Indeed, fix $\alpha\in\mathcal{C}_{i|ab}$ and two words $v_1,v_2$ of type $(c,a)$.
    Take $\tilde{\Sigma}_{v_1,\alpha}$ and $\tilde{\Sigma}_{v_2,\alpha}$ as in \eqref{eqn:sigma-rearrange}.
    Given a 2-cell $\eta:v_1\to v_2$, we then have that
    \begin{align*}
        \Psi(\eta\tensorhat\I_{[\alpha]_i})
        &=\sum_{w\in W_{cb}}\sum_{\xi\in\onb(v_1[\alpha]_i\triangleright\star)_w} \qdim(w) \Psi_{v_2,w}(\eta_{\alpha\triangleright_i\star}\xi)\Psi_{v_1,w}(\xi)^*\\
        &=\sum_{w\in W_{cb}}\sum_{\xi\in\onb((v_1\triangleright\star)\triangleleft_i\alpha)_w} \qdim(w) \Psi_{v_2,w}(\eta_{\alpha\triangleright_i\star}\tilde{\Sigma}_{v_1,\alpha}\xi)\Psi_{v_1,w}(\tilde{\Sigma}_{v_1,\alpha}\xi)^*\\
        &=\sum_{w\in W_{cb}}\sum_{\xi\in\onb((v_1\triangleright\star)\triangleleft_i\alpha)_w} \qdim(w) \Psi_{v_2,w}(\tilde{\Sigma}_{v_2,\alpha}(\eta_\star\triangleleft_i\I_\alpha)\xi)\Psi_{v_1,w}(\tilde{\Sigma}_{v_1,\alpha}\xi)^*\\
        &=\sum_{w'\in W_{ca}}\sum_{\xi'\in\onb(v_1\triangleright\star)_{w'}} \qdim(w') \Psi_{v_2,w'}(\eta_\star\xi')\Psi_{v_1,w'}(\xi')^*\otimes\I_{\Psi_i(\alpha)}\\
        &=\Psi(\eta)\otimes\I_{\Psi([\alpha]_i)}\,,
    \end{align*}
    where we used \eqref{eqn:assembly-map-right-identity} to get the fourth equality.
    This concludes the proof of the existence statement.

    The uniqueness statement in the theorem is much easier to prove (as one would expect).
    Suppose that $(g,\Psi),(g,\Psi'):\mathcal{C}\to\mathcal{D}$ are both unitary 2-functors with the stated property, i.e.\@ that there exist unitary monoidal natural transformations $\zeta^i: \Psi_i\to \Psi\Phi_i$ and $(\zeta')^i:\Psi_i\to \Psi'\Phi_i$. 
    We have to define a unitary monoidal natural transformation $\zeta:\Psi\to \Psi'$.
    Clearly $\zeta_{\emptyword}=(\Psi'^{(0)})^*\Psi^{(0)}$ is the only choice we have.
    Now suppose that we already defined $\zeta_v$ for all words of length at most $n$.
    Choose $\alpha\in\mathcal{C}_{i|ab}$ and $v'$ a word of type $(b,c)$ of length $n$.
    Then $((\zeta')^i_\alpha)^*\zeta^i_\alpha\otimes \zeta_{v'})$ is a 2-cell from $\Psi_i([\alpha]_i)\Psi(v')$ to $\Psi_i'([\alpha]_i)\Psi'(v')$.
    For $v=[\alpha]_iv'$, we can then define $\zeta_v: \Psi(v)\to \Psi'(v)$ by setting
    \[
        \zeta_v = (\Psi')^{(2)}\left[((\zeta')^i_\alpha)^*\zeta^i_\alpha\otimes \zeta_{v'}\right] (\Psi^{(2)})^*\,.
    \]
\end{proof}

As an application of the universal property, we can now show that any $C^*$-2-category satisfying the ``working definition'' of the free product is actually equivalent to the free product.
This provides a converse of sorts to Proposition~\ref{thm:poor-mans-free-product}.
\begin{proposition}
\label{thm:poor-mans-free-product-equivalence}
    Let $\mathcal{C}$ be a (strict) rigid $C^*$-2-category with 2-full sub-2-categories $(\mathcal{C}_i)_{i\in I}$ such that
    \begin{enumerate}[(i)]
        \item $\bigcup_{i\in I} B(\mathcal{C}_i)=B(\mathcal{C})$;
        \item $B(\mathcal{C}_i)\cap B(\mathcal{C}_j)=\bigcap_{i\in I} B(\mathcal{C}_i)$ for all distinct $i,j\in I$;
        \item for any choice of irreducible 1-cells $\alpha_1,\ldots,\alpha_n$ with $n\geq 0$, $\alpha_k\in \mathcal{C}_{i_k|a_kb_k}$, $a_k=b_{k-1}$, $i_k\neq i_{k+1}$ and $\alpha_k\not\cong \epsilon_{i_k|a_k}$, the composition $\alpha_1\cdots\alpha_n$ is irreducible in $\mathcal{C}_{a_1b_n}$, and all these irreducibles are distinct;
        \item tensor products of the form described in (iii) exhaust all irreducible 1-cells in $\mathcal{C}$ up to isomorphism.
    \end{enumerate}
    Writing $S=\bigcap_{i\in I} B(\mathcal{C}_i)$, the rigid $C^*$-tensor 2-category $\mathcal{C}$ together with the inclusions $\mathcal{C}_i\hookrightarrow\mathcal{C}$ satisfies the universal property of the free product of the rigid $C^*$-tensor 2-categories $(\mathcal{C}_i)_{i\in I}$ w.r.t.\@ the inclusions of $S$ into $B(\mathcal{C}_i)$ for $i\in I$.
\begin{proof}
    Throughout the proof, we identify the unit objects $\epsilon_{i|a}$ of $\mathcal{C}_{i|aa}$ with $\epsilon_a$ in $\mathcal{C}_{aa}$ for all $i\in I$ and $a\in B(\mathcal{C}_i)$.
    Let $(j_i,\iota_i)$ be the unitary 2-functor given by the inclusion of $\mathcal{C}_i$ into $\mathcal{C}$, and let $\mathcal{C}'$ be the free product constructed above.
    Denoting the associated unitary 2-embeddings $(\phi_i,\Psi_i):\mathcal{C}_i\to \mathcal{C}$, we have that $\phi_i\vert_S = \phi_j\vert_S$ by hypothesis.
    Theorem~\ref{thm:free-product-2cat-up} therefore provides a universal unitary 2-functor $(\psi,\Psi)$ from $\mathcal{C}'$ to $\mathcal{C}$.
    We claim that this is a unitary equivalence of 2-categories.

    First, we prove that $\psi$ is a bijection.
    Surjectivity is clear from condition (i).
    To prove that $\psi$ is injective, fix $x,y\in B(\mathcal{C}')$ such that $\psi(x)=\psi(y)$.
    These can be written in the form $x=\phi_i(x')$, $y=\phi_{i'}(y')$ for $i,i'\in I$ and $x'\in B(\mathcal{C}_i)$ and $y'\in B(\mathcal{C}_j)$.
    Suppose now that $\psi(x)=\psi(y)$.
    Since $\psi\phi_i=j_i$ for all $i\in I$, it follows that $x'=y'$ in $B(\mathcal{C})$.
    If $i=i'$, this immediately implies that $x=y$.
    Otherwise, condition (ii) implies that $x',y'$ lie in $S$, which also implies that $x=y$ since $\phi_i\vert_S = \phi_{i'}\vert_S$.
    In what follows, we therefore identify $B:=B(\mathcal{C}')=B(\mathcal{C})$ and treat $\psi$ as the identity.
    
    We now argue that $\Psi: \mathcal{C}'_{ab}\to \mathcal{C}_{ab}$ is fully faithful and essentially surjective everywhere.
    Faithfulness is automatic, and the essential surjectivity follows from condition (iv).
    To prove that $\Psi$ is full, it suffices to argue that $(\Psi(v),\Psi(w))\cong (v,w)$ for all $a,b\in B$, all words $v$ of type $(a,b)$ and all reduced words $w\in W_{ab}$.
    We do this by induction on the length of $v$.
    When $a=b$ and $v=\emptyword$ we have that
    \[
        (\Psi(\emptyword),\Psi(w))\cong (\epsilon_a,\Psi(w))\,.
    \]
    We already know from condition (iii) that the image of any reduced word under $\Psi$ produces an irreducible 1-cell in $\mathcal{C}$, and that all these irreducible objects are distinct.
    We conclude that both of these spaces are zero if $w$ is not the empty word, and one-dimensional otherwise.

    For the inductive step, choose $a,b,c,d\in B$, $i\in I$, $\alpha\in\mathcal{C}_{i|ab}$, $v$ a word of type $(b,c)$, $\pi\in\Irr(\mathcal{C}_{i|ad})$ and $w\in W^\ell_{i|dc}$.
    Then, by applying conditions (iii), (iv) and the induction hypothesis, we get that
    \begin{align*}
        (\Psi([\alpha]_iv), \Psi([\pi]_i:w))
        &\cong (\alpha\Psi(v),\pi\Psi(w))\\
        &\cong \bigoplus_{\substack{e\in B, w'\in W_{ec}^\ell\\ \gamma\in\Irr(\mathcal{C}_{i|be})}}
        (\Psi(v),\Psi([\gamma]_i:w'))\otimes (\alpha\Psi([\gamma]_i:w'),\pi\Psi(w))\\
        &\cong \bigoplus_{\substack{e\in B, w'\in W_{ec}^\ell\\ \gamma\in\Irr(\mathcal{C}_{i|be})}}
        (v,[\gamma]_iw')\otimes (\alpha\gamma\Psi(w'),\pi\Psi(w))\\
        &\cong \bigoplus_{\substack{e\in B, w'\in W_{ec}^\ell\\ \gamma\in\Irr(\mathcal{C}_{i|be}) \\ \pi'\in\Irr(\mathcal{C}_{i|ae})}}
        (v,[\gamma]_iw')\otimes (\alpha\gamma,\pi')\otimes (\pi'\Psi(w'),\pi\Psi(w))\,.
    \end{align*}
    It follows from condition (iii) that only the terms with $\pi'=\pi$ and $w=w'$ survive, and that the endomorphism space of $\pi\Psi(w)$ is one-dimensional.
    Combining this with the fact that $\Phi_i: \beta\mapsto [\beta]_i$ is fully faithful, we conclude that
    \begin{align*}
        (\Psi([\alpha]_iv), \Psi([\pi]_i:w))
        &\cong \bigoplus_{\gamma\in\Irr(\mathcal{C}_{i|be})}
        (v,[\gamma]_iw')\otimes (\alpha\gamma,\pi)
        \cong \bigoplus_{\gamma\in\Irr(\mathcal{C}_{i|be})}
        (v,[\gamma]_iw')\otimes ([\alpha]_i[\gamma]_i,[\pi]_i)\\
        &\cong ([\alpha]_iv,[\pi]_iw')\cong ([\alpha]_iv,[\pi_i]:w')\,.
    \end{align*}
\end{proof}
\end{proposition}

For rigid $C^*$-tensor categories, this reduces to the following simpler statement.
\begin{corollary}
\label{thm:tensor-cat-free-product-charact}
    Let $\mathcal{C}$ be a rigid $C^*$-tensor category with full rigid $C^*$-tensor subcategories $\mathcal{C}_i$ such that any tensor product of nontrivial irreducibles $\alpha_1,\ldots\alpha_n$ with $n\geq 0$, $\alpha_k\in\mathcal{C}_{i_k}$, $i_k\neq i_{k+1}$ remains irreducible in $\mathcal{C}$, and all these irreducibles are distinct.
    If these tensor products exhaust the irreducible objects in $\mathcal{C}$ up to isomorphism, then $\mathcal{C}$ together with the inclusions $\mathcal{C}_i\hookrightarrow \mathcal{C}$ satisfies the universal property of the free product of the rigid $C^*$-tensor categories $(\mathcal{C}_i)_{i\in I}$.
\end{corollary}

\section{Relation to other constructions}
Throughout, we assume that the $C^*$-2-categories we deal with are strict.

\subsection{Free product of quantum groups}
For a compact quantum group $\mathbb{G}$, denote by $\Rep_f(\mathbb{G})$ the rigid $C^*$-tensor category of finite-dimensional unitary representations of $\mathbb{G}$.
Additionally, we denote the fibre functor $\Rep_f(\mathbb{G})\to \Hilb_f$ mapping a representation of $\mathbb{G}$ to its carrier space by $F_{\mathbb{G}}$.
Combining Proposition~\ref{thm:free-product-cqg} with Woronowicz' Tannaka--Krein duality theorem \cite{woron-tkdual}, we recover a purely categorical construction of the free product of compact quantum groups.
\begin{proposition}
\label{thm:free-product-cqg}
    Let $\mathbb{G}$ and $\mathbb{H}$ be compact quantum groups and let $\mathbb{G}*\mathbb{H}$ be their free product in the sense of Wang \cite{wang-free-products}.
    Then $\Rep_f(\mathbb{G}*\mathbb{H})$ satisfies the universal property of the free product of $\Rep_f(\mathbb{G})$ and $\Rep_f(\mathbb{H})$.
    Moreover, $F_{\mathbb{G}*\mathbb{H}}$ is naturally unitarily monoidally equivalent to the fibre functor on $\Rep_f(\mathbb{G})$ induced from $F_{\mathbb{G}}$ and $F_{\mathbb{H}}$ by the universal property of the free product.
\begin{proof}
    This immediately follows from Woronowicz' Tannaka--Krein duality (see e.g.\@ the formulation in \cite[Theorem~2.3.2]{neshveyev-tuset}) after applying \cite[Theorem~3.10]{wang-free-products} and Corollary~\ref{thm:tensor-cat-free-product-charact}.
\end{proof}
\end{proposition}

\subsection{Free composition and the free product of planar algebras}

We now return to the setting of pointed $C^*$-2-categories of section~\ref{sec:pointed-cats}.
In order to define the \textit{free product} or \textit{free composition} of pointed $C^*$-2-categories $\mathcal{P}_1=(\mathcal{C}_1,a_1,b_2,u_1)$ and $\mathcal{P}_2=(\mathcal{C}_2,a_1,b_2,u_2)$, we compose $u_1$ and $u_2$ inside an appropriate free product of $\mathcal{C}_1$ and $\mathcal{C}_2$.
Here, we need the extra flexibility in Theorem~\ref{thm:free-product-2cat-up} that allows ``amalgamation'' at the level of the 0-cells: the free product we need is \textit{not} the one given by identifying $a_1=a_2$ and $b_1=b_2$ in Theorem~\ref{thm:free-product-2cat-up}.
Indeed, in that setting, there is no canonical way to compose $u_1$ and $u_2$.

Instead, we identify $b_1$ with $a_2$.
Formally, we take $S=\{*\}$ with $f_1(*)=b_1$ and $f_2(*)=a_2$ in Theorem~\ref{thm:free-product-2cat-up}, and denote the resulting free product 2-category by $\mathcal{C}$.
To simplify the notation going forward, we will simply put $B(\mathcal{C})=\{a,b,c\}$ and view $B(\mathcal{C}_1)$ and $B(\mathcal{C}_2)$ as subsets of $B(\mathcal{C})$ by identifying $a_1=a$, $b_1=a_2=b$ and $b_2=c$.
The composition of $u_1$ and $u_2$ is meaningful in $\mathcal{C}$, and produces an object $u_1u_2\in\mathcal{C}_{ac}$.

Denoting the $C^*$-2-category obtained from $\mathcal{C}$ by forgetting the 0-cell $b$ by $\mathcal{C}'$, we obtain a pointed $C^*$-2-category $\mathcal{P}_1*\mathcal{P}_2=(\mathcal{C}',a,c,u_1u_2)$ that encodes the free product planar algebra.

\begin{remark}
    The resulting object $\mathcal{P}_1*\mathcal{P}_2$ is \textit{not} the coproduct in the category of pointed $C^*$-2-categories in any meaningful sense.
    In fact, there are no canonical 2-embeddings of $\mathcal{P}_1$ and $\mathcal{P}_2$ into $\mathcal{P}_1*\mathcal{P}_2$ and the operation is asymmetric in the factors, so the term ``free product'' may cause undue confusion.
    The terminology ``free composition'' of \cite{bisch-jones-interm-subfactors} is perhaps more apt; see also Remark~\ref{rem:free-composition}.
\end{remark}

\begin{remark}
\label{rem:morita}
The 0-cell $b$ we left out is still interesting from the point of view of representation theory.
Indeed, the rigid $C^*$-tensor category of 1-cells $\mathcal{C}_{bb}$ is a free product of $\mathcal{C}_{1|bb}$ and $\mathcal{C}_{2|bb}$.
This means that the free product of $\mathcal{C}_{1|bb}$ and $\mathcal{C}_{2|bb}$ as rigid $C^*$-tensor categories lies in the same Morita equivalence class as $\mathcal{C}'_{aa}=\mathcal{C}_{aa}$ and $\mathcal{C}'_{cc}=\mathcal{C}_{cc}$.
In particular, the respective unitary representation theories of these rigid $C^*$-tensor categories are all equivalent \cite{pv-repr-subfactors,neshveyev-yamashita,ghosh-jones,psv-cohom}.
\end{remark}

\begin{remark}
\label{rem:free-composition}
At the level of subfactors, this operation on pointed $C^*$-2-categories corresponds to the free composition in the sense of \cite{bisch-jones-interm-subfactors}.
More precisely, consider a tower of finite-index subfactors $N\subset M\subset P$ of type $\mathrm{II}_1$, and let $\mathcal{P}_{N\subset M}$, $\mathcal{P}_{M\subset P}$ be the associated pointed $C^*$-2-categories put forward in Example~\ref{exa:subfactor}.
Then $\mathcal{P}_{N\subset M}*\mathcal{P}_{M\subset P}$ can be viewed as the standard invariant of a free composition of $N\subset M$ and $M\subset P$.
\end{remark}

\subsection{Free wreath products of compact quantum groups}
In this section, we explicitly need to allow $C^*$-2-categories with nonsimple tensor units in some cases.
The definition of a pointed $C^*$-category with reducible tensor units and the associated morphisms remains the same.
The free product still is only defined for $C^*$-2-categories with irreducible tensor units, but the universal property of Theorem~\ref{thm:free-product-2cat-up} remains true if the target $C^*$-2-category $\mathcal{D}$ has reducible tensor units.

In this language, the results \cite[Theorem~B]{tarrago-wahl} and \cite[Theorem~6.4.3]{jonas-phd} on the free wreath product of compact quantum groups also take on a very appealing form.
Consider a finite-dimensional $C^*$-algebra $A$ equipped with its Markov\footnote{The Markov trace on $A$ is the unique trace $\tau_A$ on $A$ such that $\tau_A(z)=\dim(zA)/\dim(A)$ for every central projection $z$ in $A$ (see \cite[\S~6]{banica-kac-subfactors} for other characterisations).} trace $\tau_A$.
Then the planar algebra associated with $A$ corresponds to the pointed $C^*$-2-category $\mathcal{P}^A=(\mathcal{C}^A,\CC,A,A)$ where the 1-cells in $\mathcal{C}^A$ are given by $\mathcal{C}^A_{a-b}=\HilbBimod_{f,a-b}$\footnote{For finite-dimensional algebras $A$ and $B$ the bimodule category $\HilbBimod_{f,A-B}$ simply consists of all Hilbert bimodules that are finite-dimensional as Hilbert spaces.} for $a,b\in\{\CC,A\}$.
In particular $\mathcal{C}^A_{\CC-\CC}=\Hilb_f$.
The distinguished object is the algebra $A$ itself, viewed as a 1-cell in $\mathcal{C}^A_{\CC-A}$.
Note that the tensor unit of $\mathcal{C}^A_{A-A}$ is simple if and only if $A$ has trivial centre (i.e.\@ $A\cong M_n(\CC)$).

There is a correspondence between (sufficiently nice) actions of quantum groups on $A$ and morphisms of pointed $C^*$-2-categories to $\mathcal{P}^A$.
This result is due to Banica and Tarrago--Wahl \cite{banica-autq,banica-coaction-planar,tarrago-wahl}.
The planar algebra formulation of \cite[Theorem~A]{tarrago-wahl} can be restated in categorical language as follows (see \cite[Remark~6.1.7]{jonas-phd}).
\begin{theorem}[Banica, Tarrago--Wahl]
\label{thm:tarrago-wahl-correspondence}
    Let $A$ be a finite-dimensional $C^*$-algebra equipped with its Markov trace $\tau_A$.  
    Then
    \begin{enumerate}[(1)]
        \item for every Kac-type compact quantum group $\mathbb{G}$ and every centrally ergodic $\tau_A$-preserving action $\alpha:\mathbb{G}\curvearrowright A$, there exists a pointed $C^*$-2-category $\mathcal{P}(\alpha)=(\mathcal{C}(\alpha),a,b,u^\alpha)$ with irreducible tensor units, together with a dimension-preserving morphism $\Phi_\alpha:\mathcal{P}(\alpha)\to\mathcal{P}^A$ of pointed $C^*$-2-categories;
        \item for every pair $(\mathcal{P},\Phi)$ consisting of a pointed $C^*$-2-category $\mathcal{P}$ with irreducible tensor units and a dimension-preserving morphism $\Phi:\mathcal{P}\to\mathcal{P}^A$, there exists a Kac-type compact quantum group $\mathbb{G}$ and a centrally ergodic $\tau_A$-preserving action $\alpha:\mathbb{G}\curvearrowright A$ such that $(\mathcal{P}(\alpha),\Phi_\alpha)$ and $(\mathcal{P},\Phi)$ are conjugate.
    \end{enumerate}

    \noindent Given actions $\alpha:\mathbb{G}\curvearrowright A$ and $\beta:\mathbb{H}\curvearrowright A$ of the type described in (1), we additionally have the following:
    \begin{enumerate}[(a)]
        \item $\alpha$ is faithful if and only if $\mathcal{P}(\alpha)$ is nondegenerate;
        \item the pairs $(\mathcal{P}(\alpha),\Phi_\alpha)$ and $(\mathcal{P}(\beta),\Phi_\beta)$ are conjugate if and only if the actions $\alpha$ and $\beta$ are conjugate.
    \end{enumerate}
\end{theorem}

The connection with Woronowicz' Tannaka--Krein duality theorem is immediately apparent in this language: if $\mathcal{P}=(\mathcal{C},a,b,u)$ is a pointed $C^*$-2-category together with a morphism $\Phi: \mathcal{P}\to \mathcal{P}^A$, looking at $\Phi$ on $\mathcal{C}_{aa}$ yields a unitary tensor functor to $\mathcal{C}^A_{\CC-\CC}=\Hilb_f$.
This is the fibre functor realising $\mathcal{C}_{aa}$ as the representation category of a quantum group $\mathbb{G}$.
The object $u\bar{u}$ in $\mathcal{C}_{aa}$ can then be viewed as a representation of $\mathbb{G}$ with carrier space $\Phi(u\bar{u})$.
Through the identifications
\[
    \Phi(u\bar{u}) \cong \Phi(u)\otimes_A \overline{\Phi(u)} \cong A\otimes_A A \cong A\,,
\]
this produces an action of $\mathbb{G}$ on $A$ with the desired properties (see also \cite[Remark~6.1.7]{jonas-phd}).

Under this correspondence, the action of the quantum automorphism group of $A$ (w.r.t.\@ the Markov trace $\tau_A$) is identified with the (essentially) canonical morphism into $\mathcal{P}^A$ from the pointed Temperley--Lieb--Jones 2-category with parameter $\delta=\sqrt{n}$ (see Example~\ref{exa:pointed-tlj}).

In this categorical language, \cite[Theorem~6.4.3]{jonas-phd} takes the following form.
\begin{theorem}[Tarrago--Wahl]
\label{thm:free-composition-wreath-product-ident}
    Let $A$, $B$ be finite-dimensional $C^*$-algebras and $\mathbb{F}$, $\mathbb{H}$ compact matrix quantum groups of Kac type.
    Consider faithful centrally ergodic $\tau_A$-preserving actions $\alpha:\mathbb{F}\curvearrowright A$ and $\beta:\mathbb{H}\curvearrowright B$.
    Then there is an associated action $\beta\wr_*\alpha$ of the free wreath product $\mathbb{G}=\mathbb{H}\wr_*\mathbb{F}$ on $A\otimes B$.
    This action is also faithful, centrally ergodic and $\tau_A$-preserving, and there is an isomorphism of pointed $C^*$-2-categories
    \begin{equation}
    \label{eqn:free-composition-wreath-product-ident}
         \mathcal{P}(\beta\wr_*\alpha) \cong \mathcal{P}(\alpha) * \mathcal{P}(\beta).
    \end{equation}
\end{theorem}
If $A$ is commutative, this reduces to \cite[Theorem~B]{tarrago-wahl}.

The isomorphism \eqref{eqn:free-composition-wreath-product-ident} conjugates the morphism of pointed $C^*$-2-categories $\Phi_{\beta\wr_*\alpha}:\mathcal{P}(\beta\wr_*\alpha)\to \mathcal{P}^{A\otimes B}$ into another morphism $\Phi:\mathcal{P}(\alpha) * \mathcal{P}(\beta)\to \mathcal{P}^{A\otimes B}$.
Using the universal property of the free product of $C^*$-2-categories, this $\Phi$ has a very concrete interpretation.
As before, write $\mathcal{P}(\alpha)=(\mathcal{C}(\alpha),a,b,u^\alpha)$ and $\mathcal{P}(\beta)=(\mathcal{C}(\beta),b,c,u^\beta)$.
Consider the $C^*$-2-category $\mathcal{D}$ with 0-cells $\{\CC,A,A\otimes B\}$, and 1-cells given by $\mathcal{D}_{a-b}=\HilbBimod_{f,a-b}$.
Then clearly $\mathcal{C}^A$ and $\mathcal{C}^{A\otimes B}$ embed into $\mathcal{D}$ in a natural way.
Similarly, $\mathcal{C}^B$ also embeds into $\mathcal{D}$, by applying the amplification functor $A\otimes-$ to objects in $\mathcal{C}^B$ (compare \cite[p.~178]{jonas-phd}).
Composing these with $\Phi_\alpha$ and $\Phi_\beta$, we obtain dimension-preserving unitary 2-functors from $\mathcal{C}(\alpha)$ and $\mathcal{C}(\beta)$ into $\mathcal{D}$.
The universal property of the free product then yields a unitary 2-functor $\Phi: \mathcal{C}(\alpha)*\mathcal{C}(\beta)\to \mathcal{D}$.
Since there are canonical unitary isomorphisms of right $A\otimes B$-modules
\[
    \Phi(u^\alpha u^\beta) \cong \Phi_\alpha(u^\alpha) \Phi_\beta(u^\beta) \cong A\otimes_A (A\otimes B) \cong A\otimes B\,,
\]
this 2-functor $\Phi$ can be viewed as a morphism of pointed 2-categories from $\mathcal{P}(\alpha)*\mathcal{P}(\beta)$ to $\mathcal{P}^{A\otimes B}$.
Stated in these terms, \cite[Theorem~6.4.3]{jonas-phd} asserts that this morphism is conjugate to $\Phi_{\beta\wr_*\alpha}: \mathcal{P}(\beta\wr_*\alpha)\to \mathcal{P}^{A\otimes B}$.

The corollary stated below is implicit in \cite{tarrago-wahl} and was also used as part of \cite[Theorem~5.2(iv)]{krvv}.
Based on the discussion in this section and Remark~\ref{rem:morita}, we can give a short, conceptual proof.

\begin{corollary}
    Let $A$ be a finite-dimensional $C^*$-algebra, equipped with its Markov trace $\tau_A$, and let $\mathbb{H}$ be a compact matrix quantum group of Kac type.
    Additionally, let $\mathbb{F}$ be a Kac-type quantum group together with a faithful, centrally ergodic and $\tau_A$-preserving action $\alpha:\mathbb{F}\curvearrowright A$.
    Put $\mathbb{G}=\mathbb{H}\wr_*\mathbb{F}$.
    Then the rigid $C^*$-tensor category $\Rep_f(\mathbb{G})$ is Morita equivalent to a free product of the form $\Rep_f(\mathbb{H})*\mathcal{C}$, where $\mathcal{C}$ is a rigid $C^*$-tensor category that is Morita equivalent to $\Rep_f(\mathbb{F})$.
\begin{proof}
    Fix a faithful, trace-preserving and centrally ergodic action $\beta$ of $\mathbb{H}$ on some finite-dimensional $C^*$-algebra $B$---such an action exists by \cite[Example~2.4]{tarrago-wahl}.

    Write $\mathcal{P}(\alpha)=(\mathcal{C}(\alpha),a,b,u^\alpha)$ and $\mathcal{P}(\beta)=(\mathcal{C}(\beta),b,c,u^\beta)$.
    As explained above, we may identify $\mathcal{P}(\beta\wr_*\alpha)$ with the free composition $\mathcal{P}(\alpha)*\mathcal{P}(\beta)$.
    In particular, this implies that $\Rep_f(\mathbb{G})$ is unitarily monoidally equivalent to $(\mathcal{C}(\alpha)*\mathcal{C}(\beta))_{aa}$, and therefore Morita equivalent to $(\mathcal{C}(\alpha)*\mathcal{C}(\beta))_{bb}$.
    At the same time, Remark~\ref{rem:morita} tells us that $(\mathcal{C}(\alpha)*\mathcal{C}(\beta))_{bb}\cong \mathcal{C}(\alpha)_{bb}*\mathcal{C}(\beta)_{bb}$ as rigid $C^*$-tensor categories.
    Moreover, $\mathcal{C}(\beta)_{bb}\cong\Rep_f(\mathbb{H})$, while $\mathcal{C}(\alpha)_{bb}$ is Morita equivalent to $\mathcal{C}(\alpha)_{aa}\cong\Rep_f(\mathbb{F})$.
    This proves the claim.
\end{proof}
\end{corollary}

The appeal of this corollary stems from the fact that many representation-theoretic properties and invariants of rigid $C^*$-tensor categories are preserved under Morita equivalence.
In this way, one can transfer results about free products to free wreath products (see e.g.\@ \cite[Corollary~D]{tarrago-wahl} and \cite[Theorem~5.2(iv)]{krvv}).

\providecommand{\bysame}{\leavevmode\hbox to3em{\hrulefill}\thinspace}
\providecommand{\href}[2]{#2}
\setlength{\bibsep}{0.0pt}

\end{document}